\documentclass[11pt]{amsart}

\usepackage{amsmath, amsthm, amssymb, amsfonts, enumerate, graphicx}
\usepackage[colorlinks=true,linkcolor=blue,urlcolor=blue]{hyperref}
\usepackage{dsfont}
\usepackage{color}
\usepackage{geometry}
\usepackage{todonotes}

\DeclareGraphicsExtensions{.png,.jpg,.bmp,.eps,.pdf}

\numberwithin{equation}{section}
\numberwithin{figure}{section}

\newtheorem{thm}{Theorem}[section]
\newtheorem{prop}[thm]{Proposition}
\newtheorem{defn}[thm]{Definition}
\newtheorem{lemma}[thm]{Lemma}

\newtheorem{assum}[thm]{Assumption}
\newtheorem{remark}[thm]{Remark}

\newcommand{\1}{\mathbf{1}}

\newcommand{\E}{\mathbb{E}}
\newcommand{\F}{\mathbb{F}}

\renewcommand{\P}{\mathbb{P}}

\newcommand{\R}{\mathbb{R}}
\renewcommand{\S}{\widehat{S}}
\newcommand{\ppp}{\mathbb{P}}

\newcommand{\aaa}{\mathcal{S}}

\newcommand{\cI}{\mathcal{I}}
\newcommand{\X}{\widehat{X}}
\newcommand{\cL}{\mathcal{L}}
\newcommand{\A}{\mathcal{A}}

\title[On the Singular Control of Exchange Rates]{On the Singular Control of Exchange Rates}

\author[Ferrari, Vargiolu]{Giorgio Ferrari, Tiziano Vargiolu}

\keywords{}

\address{G.~Ferrari: Center for Mathematical Economics (IMW), Bielefeld University, Universit\"atsstrasse 25, 33615, Bielefeld, Germany}
\email{\href{mailto:giorgio.ferrari@uni-bielefeld.de}{giorgio.ferrari@uni-bielefeld.de}}
\address{T.~Vargiolu: Dipartimento di Matematica "Tullio Levi-Civita", Universit\`{a} degli Studi di Padova, Via Trieste 63, 35100, Padova, Italy}
\email{\href{mailto:vargiolu@math.unipd.it}{vargiolu@math.unipd.it}}

\date{\today}

\numberwithin{equation}{section}

\begin{document}

\begin{abstract} 
Consider the problem of a central bank that wants to manage the exchange rate between its domestic currency and a foreign one. The central bank can purchase and sell the foreign currency, and each intervention on the exchange market leads to a proportional cost whose instantaneous marginal value depends on the current level of the exchange rate. The central bank aims at minimizing the total expected costs of interventions on the exchange market, plus a total expected holding cost. We formulate this problem as an infinite time-horizon stochastic control problem with controls that have paths which are locally of bounded variation. The exchange rate evolves as a general linearly controlled one-dimensional diffusion, and the two nondecreasing processes giving the minimal decomposition of a bounded-variation control model the cumulative amount of foreign currency that has been purchased and sold by the central bank. We provide a complete solution to this problem by finding the explicit expression of the value function and a complete characterization of the optimal control. At each instant of time, the optimally controlled exchange rate is kept within a band whose size is endogenously determined as part of the solution to the problem. 
We also study the expected exit time from the band, and the sensitivity of the width of the band with respect to the model's parameters in the case when the exchange rate evolves (in absence of any intervention) as an Ornstein-Uhlenbeck process, and the marginal costs of controls are constant. The techniques employed in the paper are those of the theory of singular stochastic control and of one-dimensional diffusions.  
\end{abstract}

\maketitle

\smallskip

{\textbf{Keywords}}: singular stochastic control; exchange rates; target zones; central bank; variational inequality; optimal stopping.

\smallskip

{\textbf{MSC2010 subject classification}}: 93E20, 60J60, 60G40, 91B64, 91G30.
\smallskip

{\textbf{OR/MS subject classification}}: Dynamic programming/optimal control; Probability: stochastic model/applications; Probability: diffusion.

\section{Introduction}

One of the main tool that a central bank has at disposal in order to maintain under control the volatility of the exchange rate is to properly purchase or sale foreign currency reserves. As a result of such interventions on the exchange market, in many cases one can observe that the exchange rate between two currencies is either kept below/above a given level, or it is maintained within announced margins on either side of a given value, the so-called central parity (or central rate). Similar regimes of the exchange rate are usually referred to as \emph{target zones}, and Switzerland, Hong Kong, and Denmark are prominent examples of countries that adopted, or adopt, such a kind of monetary policy. 

On the 6th of September 2011, the Swiss National Bank (SNB) stated in a press release \cite{NYT}: 
\vspace{0.25cm}

\emph{[...] the current massive overvaluation of the Swiss Franc poses an acute threat to the Swiss economy and carries the risk of deflationary development. The Swiss National Bank is therefore aiming for a substantial and sustained weakening of the Swiss Franc. With immediate effect, it will no longer tolerate a EUR/CHF exchange rate below the minimum rate of CHF 1.20. The SNB will enforce this minimum rate with the utmost determination and is prepared to buy foreign currency in unlimited quantities [...]}
\vspace{0.25cm}

\noindent SNB adopted such an aggressive devaluation policy until the 15th of January 2015 \cite{Economist,SNBCHF}, when SNB simply dropped its target zone policy with a very evident effect on the CHF/EUR exchange rate (see Figure \ref{EURCHF}).  

On the other hand, the 12th of January 2017 marked the 30th anniversary of the Danish central parity \cite{Danish}. The decision to pursue a fixed exchange rate policy was made in the 1980s when the Danish economy was in a crisis. Since then the Danish Krone (DKK) was anchored to the German Mark, and then, since 1999, to Euro in such a way that the Krone's central parity has been unchanged since January 12, 1987. The central rate is 7.46038 Krone per Euro, and the Krone is allowed to increase or decrease by 2.25\% (even if the fluctuations have been far smaller for many years, see Figure \ref{EURDKK}).

To end with a non-European example, as a response to the Black Saturday crisis in 1983, on October 17, 1983 the Hong Kong Dollar (HKD) has been pegged to the U.S.\ Dollar (USD), and since then the HKD/USD exchange rate is pegged to a central rate of 7.80 HKD/USD (see Figure \ref{HKDUSD}), with a band of $\pm 0.05$ HKD/USD \cite{Wiki}.

\begin{figure}[h!]
\centering
\includegraphics[scale=0.4]{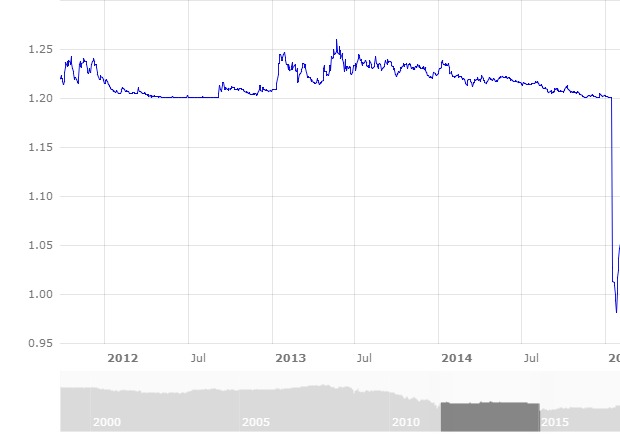}
\caption{Plot EUR/CHF exchange rate from 2011 until 2015.}
\label{EURCHF}
\end{figure}

\begin{figure}[h!]
\centering
\includegraphics[scale=0.4]{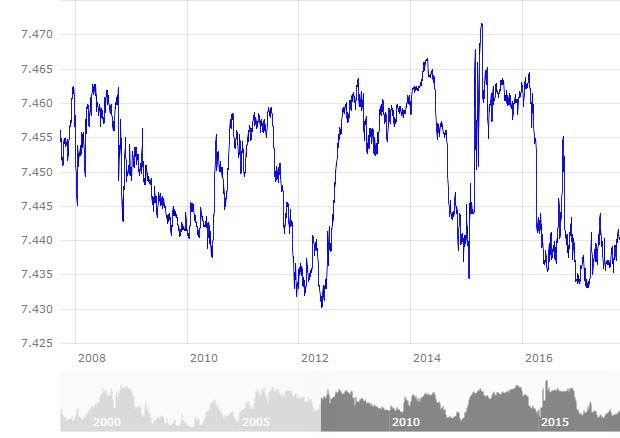}
\caption{Plot EUR/DKK exchange rate from 2008 until 2016.}
\label{EURDKK}
\end{figure}

\begin{figure}[h!]
\centering
\includegraphics[scale=0.6]{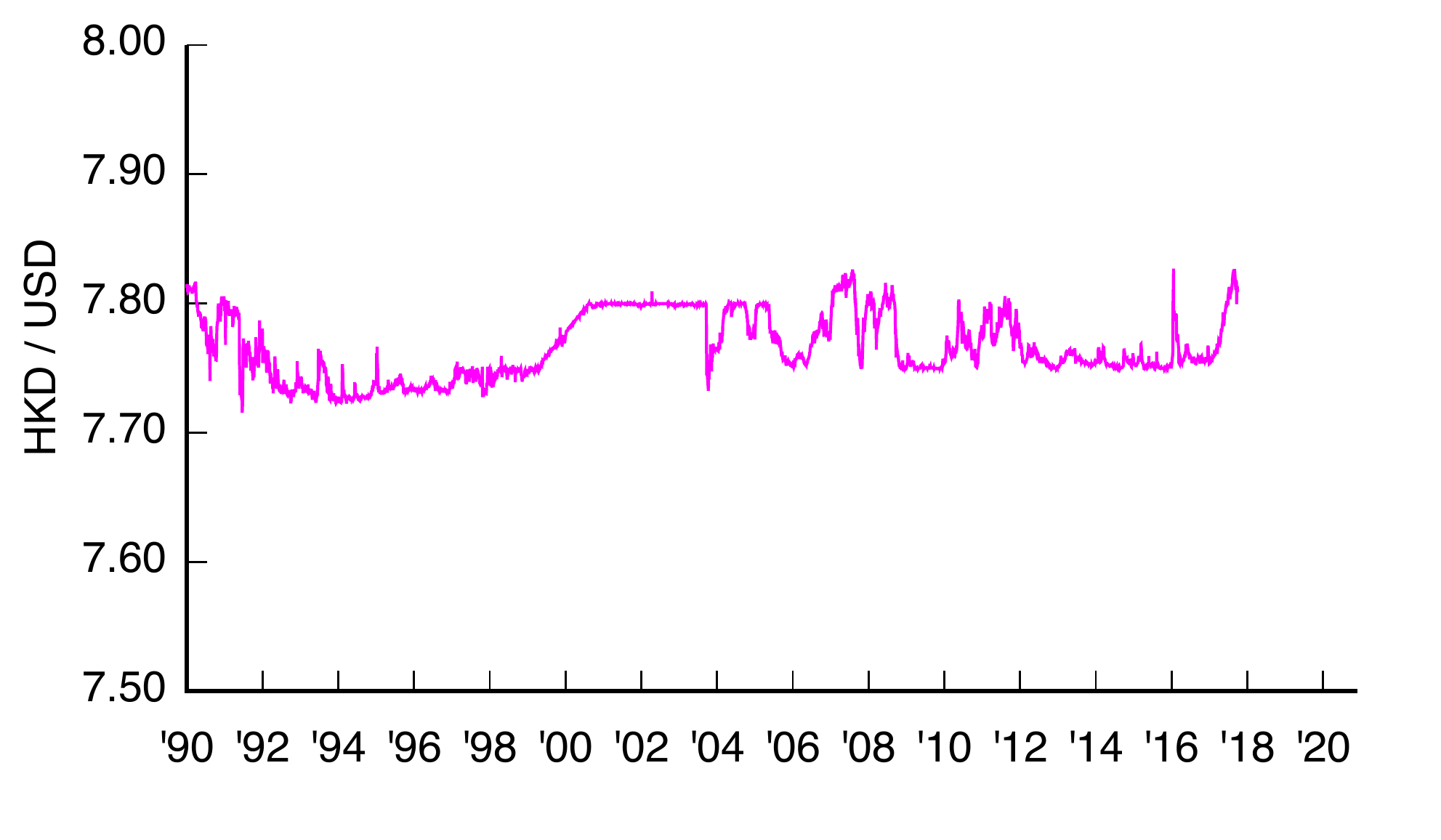}
\caption{Plot HKD/USD exchange rate from 1990 until 2017.}
\label{HKDUSD}
\end{figure}

It is not clear (nor of public knowledge) whether the width of the interval where the exchange rate is allowed to fluctuate is chosen according to some optimality criterion (e.g., maximization of social welfare or minimization of expected costs), or it is decided only on the basis of international and political agreements. In the latest years the economic and mathematical literature experienced an intensive research on target zone models. In particular, within the literature we can identify two main streams of research. On one hand, many papers develop stochastic models aiming at explaining the dynamics of exchange rates within a given target zone (see \cite{Bo, Jong, Krugman, JorMik, Larsen, Yang}, among others). Target zone models have been pioneered in \cite{Krugman} where it is assumed that the ``fundamental" (and not observed) exchange rate is a Brownian motion, which is instantaneously reflected at exogenously given upper and lower barriers: this intrinsically defines a singular stochastic control problem, whose value function is the exchange rate really observed in the market. Although in \cite{Krugman} many mathematical details are missing, in that seminal paper the author finds that the observed exchange rate is mean-reverting inside the given target zone. In the subsequent papers (see e.g.~\cite{Bo, Jong, JorMik, Larsen, Yang} and references therein), the authors assume that exchange rates fluctuate stochastically within an exogenously given interval according to a stochastic differential equation parametrized by a set of free parameters, and possibly satisfying reflecting boundary conditions, or with diffusion coefficient vanishing near the boundaries of the interval. The parameters are then calibrated in such a way that the model can fit the data on exchange rates, e.g.\ in the European monetary system.

On the other hand, several papers in the mathematical literature endogenize the width of the target zone by formulating the exchange rates' optimal management problem as a stochastic optimal control problem (see \cite{BRY, CadZap1, CadZap2, Jeanblanc, Mundaca}, and references therein). In these papers, the central bank aims at adjusting the uncertain level of the exchange rate in order to minimize the spread between the instantaneous level of the exchange rate and a given central parity. To accomplish that, the central bank can purchase or sell foreign currency, but whenever the central bank intervenes, a cost for the intervention must be paid. In those papers such a cost has both a proportional and a fixed component, thus leading to a mathematical formulation of the optimization problem as a two-sided stochastic impulsive control problem (possibly also with classical controls modeling the interventions on the domestic interest rate). It is shown that the optimally controlled exchange rate is kept within endogenously determined levels (the so-called free boundaries) and the interventions are of pure-jump type: at optimal times the exchange rate is pushed from a free boundary to another threshold level, which is also found endogenously as a part of the solution to the problem. In essence, it is optimal to follow a two-sided $(s,S)$-policy.

However, a closer look at the dynamics of the exchange rate EUR/CHF in the period 2011-2015, or at that of the exchange rate HKD/USD since 2008 reveals no jumps, but a \emph{continuous reflection} of the exchange rate at the boundaries of the interval where it is allowed to fluctuate (see Figures \ref{EURCHF} and \ref{HKDUSD}). Such an observation suggests that the optimal management problem of exchange rates might be mathematically better formulated as a singular stochastic control problem, rather than as an impulsive one. Indeed, in singular stochastic control problems the optimal control usually prescribes a continuous reflection of the controlled state variable at endogenously determined level(s) (see, e.g., Chapter VIII in \cite{FlemingSoner} and \cite{Shreve} for an introduction to singular stochastic control).

In this paper we thus introduce an infinite time-horizon, one-dimensional bounded variation singular stochastic control problem to model the exchange rates' optimal management problem. In our model, the (logarithm of the) exchange rate is a one-dimensional It\^o-diffusion satisfying a linearly controlled stochastic differential equation with suitable drift and volatility coefficients. Such general dynamics allows us to cover classical models where the exchange rate evolves as a geometric Brownian motion, as well as more realistic mean-reverting behaviors of the exchange rate's dynamics (see \cite{Sweeney, Tvedt} and references therein). The cumulative amount of purchases and sales of the foreign currency (which are the control variables of the central bank) are monotone processes, adapted to the underlying filtration, and satisfying proper integrability conditions. The central bank aims at choosing a (cumulative) purchasing-selling policy in order to minimize a total expected discounted cost functional. This is given by the sum of total expected holding costs and costs of interventions. The instantaneous holding cost of the exchange rate is measured by a general nonnegative convex function. This generalizes the quadratic cost function usually employed in the literature (cf., e.g., \cite{BRY, CadZap1}). Also, we assume that the instantaneous proportional costs of the interventions on the exchange rate depend on the current level of the exchange rate, and they are sufficiently smooth real-valued functions.

We tackle the problem via a \emph{guess-and-verify approach} by carefully employing the properties of one-dimensional regular diffusions (see, e.g., \cite{BS}), and of their excessive mappings \cite{Alvarez03}. We find that the optimal purchasing-selling policy of the central bank is triggered by two thresholds (free boundaries), which are the unique solution to a system of two coupled nonlinear algebraic equations. The optimal policy prescribes to purchase and sell the minimal amount of foreign currency that allows to keep the exchange rate within the free boundaries. Mathematically, the optimal control is given by the solution to a two-sided Skorokhod reflection problem. 

It is worth noticing that, differently from models involving impulsive controls, where the actual optimality of a candidate value function is usually proved only via numerical methods (see \cite{CadZap1, CadZap2}), here we are able to provide a complete analytical study by finding the explicit expression of the value function and of the optimal control process (up to the solution to the algebraic system for the two free boundaries). Moreover, we can provide a detailed comparative statics analysis of the free boundaries when the (log-)exchange rate (in absence of any intervention) evolves through an Ornstein-Uhlenbeck dynamics. The latter allows us to capture the mean-reverting behavior of exchange rates that has been observed in several empirical studies (see \cite{Sweeney, Tvedt} and references therein). In particular, by assuming that the instantaneous proportional costs of interventions are constant, we show that the more the exchange market is volatile, the more the central bank is reluctant to intervene. Also, we are able to numerically evaluate the expected exit times and exit probabilities from the target zone, and to relate our findings with the monetary policy adopted by the Danish Central Bank since 1987 \cite{Danish}, and by the SNB in the period 2011--2015 \cite{Economist,SNBCHF}.

The contribution of this paper is twofold. On the one hand, we contribute to the literature from the modeling point of view. Indeed, by introducing a singular stochastic control problem to model the exchange rates' optimal management problem faced by a central bank, we are able to mimick the continuous reflection of the exchange rate at the target zone's boundaries which seems to happen in reality (see Figures \ref{EURCHF} and \ref{HKDUSD}). From the mathematical point of view, we contribute by providing the explicit solution to a bounded variation singular stochastic control in a very general setting with state variable evolving as a general one-dimensional diffusion, and with instantaneous marginal costs of control that are state-dependent. To the best of our knowledge, the explicit solution to a similar problem is not available in the literature yet. 

The work that is perhaps closest to ours is \cite{Matomaki}, where a one-dimensional, bounded variation singular stochastic control problem over an infinite time-horizon has been studied. However, one can come across several major differences between our paper and \cite{Matomaki}. First of all, in \cite{Matomaki} the instantaneous marginal costs of control are constant. Second of all, in \cite{Matomaki} the state dynamics (in the notation of that paper) is $Z_t = X_t + U_t - D_t$, where $X$ is an uncontrolled one-dimensional regular diffusion, and $(U,D)$ gives the minimal decomposition of a process of bounded variation. In our paper, instead, the dynamics of the state variable is given in differential form (see \eqref{stateX}), and, differently to \cite{Matomaki}, the controlled state process at time $t\geq 0$ cannot be written as the sum of an uncontrolled one and of the cumulative bounded variation control exerted up to time $t$. Finally, in \cite{Matomaki} the optimal control is sought within the class of barrier policies, whereas we here obtain optimality in a larger class (see our Definition \ref{def:admiss} below).

The rest of the paper is organized as follows. In Section \ref{sec:setting} we set up the probabilistic setting, whereas in Section \ref{sec:problem} we introduce the exchange rates' optimal management problem that is the object of our study. In Section \ref{sec:solving} we solve the problem by proving first a preliminary verification theorem, and then constructing the value function and the optimal control. In Section \ref{sec:OU} we assume that the (log-)exchange rate is an Ornstein-Uhlenbeck process, and we provide the sensitivity of the free boundaries with respect to the model's parameters and a study of the expected hitting time at the free boundaries. Finally, in the appendix we collect some auxiliary results needed in the paper.


\section{Setting and Problem formulation}
\label{sec:problemsetting}

\subsection{The Probabilistic Setting}
\label{sec:setting}

Let $(\Omega,\mathcal{F},\ppp)$ be a complete probability space, $B$ a one-dimensional Brownian motion, and denote by $\F = (\mathcal{F}_t)_{t \geq 0}$ a right-continuous filtration to which $B$ is adapted. We introduce the nonempty sets
\begin{align}
\aaa := & \{ \nu:\Omega \times \R_+ \to \R_+, \mbox{ $\F$-adapted and such that } t \mapsto \nu_t \mbox{ is a.s.} \label{setA} \\
& \hspace{1cm} \mbox{(locally) of bounded variation, left-continuous and s.t. } \nu_0 = 0 \},\nonumber \\
\mathcal{U} := & \{\vartheta: \vartheta \in \aaa\,\,\mbox{ and }  t \mapsto \vartheta_t \mbox{ is nondecreasing}\}. \label{setU}
\end{align}
Then, for any $\nu \in \aaa$, we denote by $\xi, \eta \in \mathcal{U}$ the two processes providing the minimal decomposition of $\nu$; that is, such that
$$ \nu_t = \xi_t - \eta_t, \quad t \geq 0,$$
and the increments $\Delta\xi_t = \xi_{t+} - \xi_t$ and $\Delta\eta_t:=\eta_{t+}-\eta_t$ are supported on disjoint subsets of $\R_+$. In the following, we set $\xi_0  = \eta_0 = 0$ a.s., without loss of generality, and for frequent future use we notice that any $\nu \in \aaa$ satisfies
$$\nu_t = \nu^c_t + \nu^j_t, \qquad t \geq 0.$$
Here $\nu^c$ is the continuous part of $\nu$, and the jump part $\nu^j$ is such that $\nu^j_t:= \sum_{0 \leq s < t}\Delta \nu_s$, where $\Delta\nu_t := \nu_{t +} - \nu_t$, $t\geq0$.

We then consider on $(\Omega,\mathcal{F},\ppp)$ a process $X$ satisfying the following stochastic differential equation (SDE) 
\begin{equation}
\label{stateX}
dX_t = \mu(X_t) dt + \sigma(X_t) dB_t + d\xi_t - d\eta_t, \quad X_0 = x \in \cI.
\end{equation}
Here $\cI := (\underline x, \overline x)$, with $- \infty \leq \underline x < \overline x \leq + \infty$, and $\mu$ and $\sigma$ are suitable drift and diffusion coefficients. The process $X$ represents the (log-)exchange rate between two currencies. The drift coefficient $\mu$ measures the trend of the exchange rate, whereas $\sigma$ the fluctuations around this trend. The central bank can adjust the level of $X$ through the processes $\xi$ and $\eta$. In particular, $\xi_t$ and $\eta_t$ could be an indication of the cumulative amount of the foreign currency which has been bought or sold up to time $t\geq 0$ in order to push the level of the exchange rate up or down, respectively.

The following assumption ensures that, for any $\nu \in \aaa$, there exists a unique strong solution to \eqref{stateX} (see \cite{Protter}, Theorem V.7).
\begin{assum}
\label{A1}
The coefficients $\mu:\R \to \R$ and $\sigma:\R \to (0,\infty)$ belong to $C^1(\R)$. Moreover, there exists $L>0$ such that for all $x,y \in \cI$, 
$$ |\mu(x) - \mu(y)| + |\sigma(x) - \sigma(y)| \leq L|x - y|. $$
\end{assum}

From now on, in order to stress its dependence on the initial value $x \in \cI$ and on the two processes $\xi$ and $\eta$, we refer to the (left-continuous) solution to \eqref{stateX} as $X^{x;\xi,\eta}$, where appropriate. Also, in the rest of the paper we use the notation $\E_x[f(X^{\xi,\eta}_t)]=\E[f(X^{x,\xi,\eta}_t)]$. Here $\E_x$ is the expectation under the measure $\ppp_x(\,\cdot\,):=\ppp(\,\cdot\,|X^{\xi,\eta}_0=x)$ on $(\Omega,\mathcal{F})$, and $f:\R \to \R$ is any Borel-measurable such that $f(X^{\xi,\eta}_t)$ is integrable.

We also denote by
$$ \sigma_{\cI} := \inf\{ t \geq 0\ |\ X_t^{x;\xi,\eta} \notin \cI \} $$
the first time when the controlled process $X_t^{x;\xi,\eta}$ leaves $\cI$.

We also consider a one-dimensional diffusion evolving according to the SDE
\begin{equation}
\label{statehatX}
d\widehat X_t = [ \mu(\widehat X_t) + (\sigma \sigma')(\widehat X_t)] dt + \sigma(\widehat X_t) d\widehat B_t, \quad \widehat X_0 = x \in \cI.
\end{equation}
Notice that, under Assumption \ref{A1}, there exists a weak solution $(\widehat{\Omega}, \widehat{\mathcal{F}}, \widehat{\F}, \widehat{\P}_x,\widehat B,\widehat{X})$ of \eqref{statehatX} that is unique in law, up to a possible explosion time (see Chapter 5.5 in \cite{KS}, among others). Indeed, under Assumption \ref{A1} one has that for any $x \in \cI$ there exists $\epsilon_o > 0$ such that 
\begin{equation}
\label{LIhatX}
\int_{x - \epsilon_o}^{x + \epsilon_o} \frac{1 + |\mu(z)| + |\sigma \sigma'(z)|}{|\sigma^2(z)|}\ dz < + \infty.
\end{equation}

\noindent We shall consider such a solution fixed for any initial condition $x\in \mathcal{I}$ throughout this paper. Moreover, \eqref{LIhatX} guarantees that $\widehat X$ is a regular diffusion. That is, starting from $x \in \cI$, $\widehat X$ reaches any other $y \in \cI$ in finite time with positive probability. Finally, to stress the dependence of $\X$ on its initial value, from now on we write $\X^x$, where needed, and we denote by $\widehat{\E}_x$ the expectation under the measure $\widehat{\P}_x$.

\begin{remark}
\label{rem:Girsa}
Define the new measure $\mathbb{Q}_x$ through the Radon-Nikodym derivative 
\begin{align}
\label{zeta}
Z_t:=\frac{d\mathbb{Q}_x}{d\ppp_x}\bigg|_{\mathcal{F}_t}=\exp\Big\{\int^t_0\sigma'(X^{0,0}_s)dB_s-\frac{1}{2}\int^t_0(\sigma')^2(X^{0,0}_s)ds\Big\},\qquad\P_x-\text{a.s.},
\end{align}
which is an exponential martingale by the boundedness of $\sigma'$. Then by Girsanov theorem the process 
\begin{equation}
\label{wideB}
\widehat{B}_t:=B_t-\int_0^t\sigma'(X^{0,0}_s)ds \quad \mbox{ is a standard Brownian motion under $\mathbb{Q}_x$,}
\end{equation}
and it is not hard to verify that $\text{Law}\,(X^{0,0}\big|\mathbb{Q}_x)=\text{Law}\,(\X\big|\widehat{\ppp}_x)$.  
\end{remark}

The infinitesimal generator of the uncontrolled diffusion $X^{x;0,0}$ is denoted by $\cL_{X}$ and is defined as
\begin{align}
\label{eq:LX}
(\cL_{X} f)\,(x):=\frac{1}{2}\sigma^2(x)f''(x)+\mu(x)f'(x),\quad f\in C^2(\overline{\cI}),\, x\in\cI,
\end{align}
whereas the one of $\X$ is denoted by $\cL_{\X}$ and is defined as
\begin{align}
\label{eq:LXhat}
(\cL_{\X} f)\,(x):=\frac{1}{2}\sigma^2(x)f''(x)+(\mu(x)+\sigma(x)\sigma'(x))f'(x),\quad f\in C^2(\overline{\cI}), x\in\cI.
\end{align}
Letting $r>0$ be a fixed constant, we make the following standing assumption.
\begin{assum}
\label{ass:rate}
$r-\mu'(x) >0$ for $x\in\overline{\cI}$.
\end{assum}
\noindent In the subsequent optimization problem, the parameter $r>0$ will play the role of the central bank's discount factor (see \eqref{costfunct} below). 

We introduce $\psi$ and $\phi$ as the fundamental solutions of the ordinary differential equation (ODE) (see Ch.~2, Sec.~10 of \cite{BS}),
\begin{align}
\label{ODE}
\cL_X u(x)-ru(x)=0,\qquad x\in\cI,
\end{align}
and we recall that they are strictly increasing and decreasing, respectively. For an arbitrary $x_0 \in \mathcal{I}$ we also denote by 
$$ S'(x) := \exp\left(- \int_{x_0}^x \frac{2\mu(z)}{\sigma^2(z)}\ dz \right), \qquad x\in\cI, $$
the derivative of the scale function of $(X^{x;0,0}_t)_{t\ge 0}$, and by $W$ the constant Wronskian
\begin{equation}
\label{WronskianX}
W:= \frac{{\psi}'(x){\phi}(x) - {\phi}'(x){\psi}(x)}{S'(x)}, \quad x \in \cI.
\end{equation}

Moreover, under Assumption \ref{A1}, any solution to the ODE
\begin{align}
\label{ODE2}
\cL_{\X} u(x)-(r-\mu'(x))u(x)=0,\qquad x\in\cI,
\end{align}
can be written as a linear combination of the fundamental solutions $\widehat{\psi}$ and $\widehat{\phi}$, which again by \cite[Chapter 2.10]{BS} are strictly increasing and decreasing, respectively. 
Finally, letting $x_0 \in \mathcal{I}$ to be arbitrary, we denote by 
$$ \S'(x) := \exp\left(- \int_{x_0}^x \frac{2\mu(z) + 2 \sigma (z) \sigma'(z)}{\sigma^2(z)}\ dz \right), \qquad x\in\cI, $$ 
the derivative of the scale function of $(\X^x_t)_{t\ge 0}$, by 
\begin{equation}
\label{hatm}
\widehat{m}'(x) := \frac{2}{\sigma^2(x)\, \S'(x)},
\end{equation}
the density of the speed measure of $(\X^x_t)_{t\ge 0}$, and by $w$ the Wronskian
\begin{equation}
\label{Wronskian}
w:= \frac{\widehat{\psi}'(x)\widehat{\phi}(x) - \widehat{\phi}'(x)\widehat{\psi}(x)}{\S'(x)}, \quad x \in \cI.
\end{equation}

\begin{remark}
It is easy to see that the scale functions and speed measures of the two diffusions $X^{x;0,0}$ and $\X^x$ are related through $\S'(x)=S(x)/\sigma^2(x)$ and $\widehat{m}'(x)=2/S'(x)$ for $x \in \cI$.
\end{remark}

Concerning the boundary behavior of the real-valued It\^o-diffusions $X^{x;0,0}$ and $\widehat X$, in the rest of this paper we assume that $\underline x$ and $\overline x$ are natural for those two processes (see \cite{BS} for a complete discussion of the boundary behavior of one-dimensional diffusions). This in particular means that they are unattainable in finite time and that
\begin{equation}
\label{psiphiproperties1}
\lim_{x \downarrow \underline{x}}\psi(x) = 0,\,\,\,\,\lim_{x \downarrow \underline{x}}\phi(x) = + \infty,\,\,\,\,\lim_{x \uparrow \overline{x}}\psi(x) = + \infty,\,\,\,\,\lim_{x \uparrow \overline{x}}\phi(x) = 0,
\end{equation}
\begin{equation}
\label{psiphiproperties2}
\lim_{x \downarrow \underline{x}}\frac{\psi'(x)}{S'(x)} = 0,\,\,\,\,\lim_{x \downarrow \underline{x}}\frac{\phi'(x)}{S'(x)} = -\infty,\,\,\,\,\lim_{x \uparrow \overline{x}}\frac{\psi'(x)}{S'(x)} = + \infty,\,\,\,\,\lim_{x \uparrow \overline{x}}\frac{\phi'(x)}{S'(x)} = 0,
\end{equation}
and
\begin{equation}
\label{psiphiproperties1bis}
\lim_{x \downarrow \underline{x}}\widehat{\psi}(x) = 0,\,\,\,\,\lim_{x \downarrow \underline{x}}\widehat{\phi}(x) = + \infty,\,\,\,\,\lim_{x \uparrow \overline{x}}\widehat{\psi}(x) = + \infty,\,\,\,\,\lim_{x \uparrow \overline{x}}\widehat{\phi}(x) = 0,
\end{equation}
\begin{equation}
\label{psiphiproperties2bis}
\lim_{x \downarrow \underline{x}}\frac{\widehat{\psi}'(x)}{\S'(x)} = 0,\,\,\,\,\lim_{x \downarrow \underline{x}}\frac{\widehat{\phi}'(x)}{\S'(x)} = -\infty,\,\,\,\,\lim_{x \uparrow \overline{x}}\frac{\widehat{\psi}'(x)}{\S'(x)} = + \infty,\,\,\,\,\lim_{x \uparrow \overline{x}}\frac{\widehat{\phi}'(x)}{S'(x)} = 0.
\end{equation}

In the following we also make the next standing assumption.
\begin{assum}
\label{ass:psiprimephiprime}
One has $\lim_{x\downarrow \underline x}\phi'(x)=-\infty$ and $\lim_{x\uparrow \overline x}\psi'(x)=\infty$.
\end{assum}

Under Assumption \ref{ass:psiprimephiprime} we show in Lemma \ref{lem:AM} in the Appendix that one has $\widehat{\phi}=-\phi'$ and $\widehat{\psi}=\psi'$ (see also the second part of the proof of Lemma 4.3 in \cite{AlvarezMatomaki}).

\begin{remark}
It is worth noticing that all the assumptions that we have made regarding the diffusions $X^{x;0,0}$ and $\widehat X$ (namely, Assumptions \ref{A1}, \ref{ass:rate} and \ref{ass:psiprimephiprime}) are satisfied in the relevant cases of a (log-)exchange rate given, e.g., by a drifted Brownian motion (i.e.\ $\mu(x) = \mu>0$ and $\sigma(x)=\sigma>0$), or by a mean-reverting process (i.e.\ $\mu(x) = \theta(\mu-x)$, for some constants $\theta>0$, $\mu\in \R$ and $\sigma(x)=\sigma>0$), both defined on $\cI = \R$, i.e.\ with $\underline x = - \infty$, $\bar x = + \infty$.
\end{remark}

For future reference, for all $x,y \in \cI$ we introduce the Green functions associated to the diffusion $X^{x;0,0}$
\begin{equation}
\label{Green}
{G}(x,y):= W^{-1} \cdot \left\{
\begin{array}{ll}
{\psi}(x){\phi}(y), & x \leq y,\\[+4pt]
{\phi}(x){\psi}(y), & x \geq y,
\end{array}
\right.
\end{equation}
and to the diffusion $\X^x$
\begin{equation}
\label{Greenhat}
\widehat{G}(x,y):= w^{-1} \cdot \left\{
\begin{array}{ll}
\widehat{\psi}(x)\widehat{\phi}(y), & x \leq y,\\[+4pt]
\widehat{\phi}(x)\widehat{\psi}(y), & x \geq y.
\end{array}
\right.
\end{equation}
Then one has that the resolvents
\begin{equation}
\label{resolvent}
({R} f)(x) := {\E}_x\bigg[ \int_0^\infty e^{-rs} f(X^{0,0}_s)\ ds \bigg], \quad x \in \cI, 
\end{equation}
and
\begin{equation}
\label{resolventhat}
(\widehat{R} f)(x) := \widehat{\E}_x\bigg[ \int_0^\infty e^{- \int_0^s (r-\mu'(\X_u)) du} f(\X_s)\ ds \bigg], \quad x \in \cI, 
\end{equation}
which are defined for any function $f$ such that the previous expectations are finite, admit the representations
\begin{equation}
\label{resolventrepr}
(R f)(x) = {\E}_x\bigg[\int_0^{\infty} e^{-rs} f(X^{0,0}_s) ds\bigg] = \int_{\cI} f(y){G}(x,y){m}'(y) dy, 
\end{equation}
and
\begin{equation}
\label{resolventhatrepr}
(\widehat{R} f)(x) = \widehat{\E}_x\bigg[\int_0^{\infty} e^{- \int_0^s (r-\mu'(\X_u)) du} f(\X_s) ds\bigg] = \int_{\cI} f(y)\widehat{G}(x,y)\widehat{m}'(y) dy, 
\end{equation}
for all $x \in \cI$. 
Notice that $Rf$ and $\widehat{R}f$ solve the ODEs
\begin{equation}
\label{resolventsODE}
\big(\cL_X - r\big)(Rf)(x)=-f(x), \qquad \big(\cL_{\X} - (r-\mu'(x))\big)(\widehat{R}f)(x) = -f(x),
\end{equation}
for any $x \in \cI$. Moreover,
\begin{equation}
\label{relation-resolvents}
(Rf)'(x) = (\widehat{R}f')(x), \qquad x \in \cI,
\end{equation}
for any $f \in C^1(\cI)$ such that $Rf$ and $\widehat{R}f'$ are well defined (a proof of relation \eqref{relation-resolvents} can be found in the appendix for the sake of completeness).

Finally, the following useful equations hold for any $\underline{x}<\alpha<\beta<\overline{x}$ (cf.~par.~10, Ch.~2 of \cite{BS}):
\begin{equation}
\label{psiphiproperties3}
\left\{ \begin{array}{ll}
\displaystyle \frac{\widehat{\psi}'(\beta)}{\S'(\beta)} - \frac{\widehat{\psi}'(\alpha)}{\S'(\alpha)}= \int_{\alpha}^{\beta}\widehat{\psi}(y)(r - \mu'(y))\widehat{m}'(y) dy, \vspace{0.25cm} \\
\displaystyle \frac{\widehat{\phi}'(\beta)}{\S'(\beta)} - \frac{\widehat{\phi}'(\alpha)}{\S'(\alpha)}= \int_{\alpha}^{\beta}\widehat{\phi}(y)(r - \mu'(y))\widehat{m}'(y) dy.	
\end{array} \right.
\end{equation}


\subsection{The Optimal Control Problem}
\label{sec:problem}

In this section we introduce the optimization problem faced by the central bank. The central bank can adjust the level of the exchange rate by purchasing or selling one of the two currencies (i.e.\ by properly exerting $\xi$ and $\eta$), and we suppose that a policy of currency's devaluation or evaluation results into proportional costs, $c_1$ and $c_2$, that depend on the current level of the exchange rate. Also, we assume that, being $X_t$ the level of the (log-)exchange rate at time $t\geq 0$, the central bank faces an holding cost $h(X_t)$. 

The total expected cost associated to a central bank's policy $\nu\in \aaa$ is therefore
\begin{equation}
\label{costfunct}
\mathcal{J}_x(\nu) := \E_x\bigg[ \int_0^{\sigma_{\cI}} e^{- r s} h(X_s^{\xi,\eta})\ ds +  \int_0^{\sigma_{\cI}} e^{- r s} \Big(c_1(X_s^{\xi,\eta})\, {\scriptstyle{\oplus}}\, d\xi_s +   c_2(X_s^{\xi,\eta})\, {\scriptstyle{\ominus}}\, d\eta_s\Big) \bigg].
\end{equation}
In \eqref{costfunct} $r>0$ is a suitable discount factor of the central bank,
\begin{align}
\label{defintegral1}
& \int_0^{\sigma_\cI} e^{- r s} c_1(X_s^{x,\xi,\eta})\,{\scriptstyle{\oplus}}\,d\xi_s :=\int_0^{\sigma_I} e^{- r s} c_1(X_s^{x,\xi,\eta})\ d\xi^c_s \nonumber \\
& \hspace{0.5cm} + \sum_{s < \sigma_\cI} e^{- r s} \int_0^{\Delta \xi_s} c_1(X_s^{\xi,\eta} + z)\ dz, 
\end{align}
and
\begin{align}
\label{defintegral2}
& \int_0^{\sigma_\cI} e^{- r s} c_2(X_s^{x,\xi,\eta})\,{\scriptstyle{\ominus}}\,d\eta_s:= \int_0^{\sigma_I} e^{- r s} c_2(X_s^{x,\xi,\eta})\ d\eta^c_s \nonumber \\
&\hspace{0.5cm} + \sum_{s < \sigma_\cI} e^{- r s} \int_0^{\Delta \eta_s} c_2(X_s^{\xi,\eta} - z)\ dz,
\end{align}
and $\xi^c$ and $\eta^c$ denote the continuous parts of $\xi$ and $\eta$, respectively. Notice that the definition of the costs of control as in \eqref{defintegral1} and \eqref{defintegral2} 
has been introduced in \cite{Zhu92}, and it is now common in the singular stochastic control literature (see \cite{LZ11}, among many others).

Regarding the holding cost $h$ and the proportional costs $c_i$, we suppose the following.
\begin{assum} 
\label{A2}
\begin{itemize}\hspace{10cm}
\item[(i)] $h:\R \to [0,+ \infty)$ belongs to $C^1(\cI)$;
\item[(ii)] For any $i=1,2$, $c_i:\R \to \R$ belongs to $C^2(\cI)$. Moreover, setting $\widehat{c}_i := (\cL_{\X} - (r - \mu')) c_i$, $i = 1,2$, we have
$$ - \widehat{c}_1(x) + h'(x) \left\{ \begin{array}{ll}
	< 0,	& x < \widetilde{x}_1, \\
	= 0,	& x = \widetilde{x}_1, \\
	> 0,	& x > \widetilde{x}_1,
	\end{array} \right. $$
$$ \widehat{c}_2(x) + h'(x) \left\{ \begin{array}{ll}
	< 0,	& x < \widetilde{x}_2, \\
	= 0,	& x = \widetilde{x}_2, \\
	> 0,	& x > \widetilde{x}_2,
	\end{array} \right. $$
for some $\widetilde{x}_1$, $\widetilde{x}_2$ such that $\underline x < \widetilde{x}_1 < \widetilde{x}_2 < \overline x$. Furthermore, 
$$c_1(x) + c_2(x) > 0, \quad x \in \overline{\cI},$$ 
$$\widehat{c}_1(x) + \widehat{c}_2(x) < 0, \quad x \in {\cI},$$
and the representation
\begin{equation}
\label{repr-c}
c_i(x) = - \widehat{\E}_x\bigg[\int_0^{\infty} e^{-\int_0^s (r - \mu'(\X_u))du}\, \widehat{c}_i(\X_s) ds\bigg] = - (\widehat{R}\,\widehat{c}_i)(x), \quad x \in \cI,
\end{equation}
holds true.
Finally, there exists $K_i > 0$ and $\gamma \geq 1$ such that
$$ |c_i(x)| \leq K_i (1 + |x|^\gamma), \quad x \in \cI. $$
\end{itemize}
\end{assum}

\begin{remark}
\begin{enumerate}\hspace{10cm}
\item All the results of this paper also hold for a slighly weaker regularity condition on $c_i$, $i=1,2$; namely, if $c_i \in W^{2,\infty}_{loc}(\cI)$. The latter is equivalent by Sobolev's embeddings (see, e.g., Cor.~9.15 in Ch.~9 of \cite{Br}) to assuming that, for any $i=1,2$, $c_i$ is continuously differentiable with second derivative which is locally bounded in $\cI$.
\item It is easy to verify that, for example, $h(x) = \frac12 (x - \theta)^2$, $\theta\in \R$, and $c_i(x)=c_i>0$ for all $x \in \cI$ satisfy Assumption \ref{A2}. 
\item It is worth noticing that \eqref{repr-c} is in essence an integrability condition. Indeed, if the trasversality condition 
\begin{equation*}
\lim_{t \to + \infty} \widehat{\mathbb{E}}_x\left[ e^{-\int_0^t (r -\mu'(\X_s)) ds} c_i(\X_t) \right] = 0, \qquad i=1,2,
\end{equation*}
holds true and 
$$\widehat{\mathbb{E}}_x\bigg[\int_0^{\infty}e^{-\int_0^s (r - \mu'(\X_u))du}\, |\widehat{c}_i(\X_s)| ds\bigg] < \infty,$$
then an application of Dynkin's formula (up to a standard localization argument) gives \eqref{repr-c}.
\end{enumerate}
\end{remark}

The following definition characterizes the class of admissible controls. 

\begin{defn}
\label{def:admiss}
For any $x \in \cI$ we say that $\nu \in \aaa$ is an {\rm admissible control}, and we write $\nu \in \A(x)$, if $X^{x,\xi,\eta}_t \in \cI$ for all $t > 0$ (i.e., $\sigma_{\cI}=+\infty$ $\P_x$-a.s.) and the following hold true:
\begin{itemize}
\item[(a)]  $ \displaystyle \E_x\bigg[\int_0^{\infty} e^{-rs} |c_1(X_s^{\xi,\eta})|\,{\scriptstyle{\oplus}}\,d\xi_s +  \int_0^{\infty} e^{-rs} |c_2(X_s^{\xi,\eta})|\,{\scriptstyle{\ominus}}\,d\eta_s \bigg] < + \infty$; \vspace{0.25cm}

\item[(b)]   $ \displaystyle \E_x\bigg[ \int_0^{\infty} e^{-rs} h(X_s^{\xi,\eta})\ ds \bigg] < + \infty$; \vspace{0.25cm}

\item[(c)]   $ \displaystyle \E_x\Big[ \sup_{t \geq 0} e^{- \frac{r}{2} t} |X_t^{\xi,\eta}|^{1 + \gamma} \Big] < + \infty$ \qquad \text{(for $\gamma$ as in Assumption \ref{A2}-(ii))}.
\end{itemize}
\end{defn}
 

The central bank aims at picking an admissible $\nu^{\star}$ such that the total expected cost functional \eqref{costfunct} is minimized; that is, it aims at solving
\begin{equation}
\label{value}
v(x) := \inf_{\nu \in \A(x)} \mathcal{J}_x(\nu), \qquad x \in \cI.
\end{equation}

Problem \eqref{value} takes the form of a singular stochastic control problem (see, e.g., \cite{Shreve} for an introduction); that is, a problem where the (random) measure on $\R_+$ induced by a control process might be singular with respect to the Lebesgue measure. 


\section{Solving the Problem}
\label{sec:solving}

\subsection{A Preliminary Verification Theorem}
\label{vertheo}

In this section we prove a verification theorem, which provides a set of sufficient conditions under which a candidate value function and a candidate control process are indeed optimal. To this end, we notice that according to the classical theory of singular stochastic control (see, e.g., Chapter VIII of \cite{FlemingSoner}), we expect $v$ to identify with a suitable solution to the Hamilton-Jacobi-Bellman (HJB) equation
\begin{equation}
\label{HJB}
\min\Big\{\big(\cL_X - r\big)u(x) + h(x), c_2(x) - u'(x), u'(x) + c_1(x)\Big\}= 0, \qquad x \in \cI.
\end{equation}
In fact, the latter takes the form of a variational inequality with state-dependent gradient constraints.

\begin{thm}[Verification Theorem]
\label{thm:verifico}
Suppose that Assumption \ref{A2} holds true and assume that the Hamilton-Jacobi-Bellman equation \eqref{HJB} admits a $C^2$ solution $u:\cI \to \R$ such that
\begin{equation*}
\label{growthcandidate}
|u(x)| \leq K(1 + |x|^{1+\gamma}), \quad x \in \cI,
\end{equation*}
for some $K>0$, and where $\gamma \geq 1$ is the growth coefficients of $c_i$, $i=1,2$ (see Assumption \ref{A2}-(ii)). Then one has that $u \leq v$ on $\cI$.

Moreover, given an initial condition $x \in \cI$, suppose also that there exists $\widehat{\nu} \in \A(x)$ such that the processes $\widehat{\xi}$ and $\widehat{\eta}$ providing its minimal decomposition are such that
\begin{equation}
\label{insideC}
X^{x,\widehat{\xi},\widehat{\eta}}_t \in \Big\{x \in \cI:\, \big(\cL_X - r\big)u(x) + h(x) =0\Big\},
\end{equation}
Lebesgue-a.e.\ $\P$-a.s., the process
\begin{equation}
\label{mg}
\bigg(\int_0^t  e^{-rs} \sigma(X^{x;\widehat{\xi},\widehat{\eta}}_s) u'(X^{x;\widehat{\xi},\widehat{\eta}}_s)\ dB_s\bigg)_{t\geq0} \quad \mbox{is an $\mathbb{F}$-martingale},
\end{equation}
and
\begin{equation}
\label{flatoff}
\left\{ \begin{array}{ll}
\displaystyle \int_0^T \big(u'(X^{x,\widehat{\xi},\widehat{\eta}}_t) + c_1(X^{x,\widehat{\xi},\widehat{\eta}}_t)\big)\,{\scriptstyle{\oplus}}\,d\widehat{\xi}_t =0,\\ \\
\displaystyle \int_0^T \big(c_2(X^{x,\widehat{\xi},\widehat{\eta}}_t) - u'(X^{x,\widehat{\xi},\widehat{\eta}}_t)\big)\,{\scriptstyle{\ominus}}\,d\widehat{\eta}_t =0,
\end{array} 
\right.
\end{equation}
for all $T\geq 0$ $\P$-a.s. Then $u = v$ on $\cI$ and $\widehat{\nu}$ is optimal for \eqref{value}.
\end{thm}

\begin{proof}
The proof is organized in two steps. We first prove that $u \leq v$ on $\cI$, and then that $u \geq v$ on $\cI$, and $\widehat{\nu}$ is optimal for \eqref{value}.
\vspace{0.25cm}

\emph{Step 1.} Let $x \in \cI$ and $\nu \in \A(x)$. Since $u \in C^2(\cI)$ we can apply It\^o-Meyer's formula for semimartingales  (see \cite{Meyer}, pp.\ 278--301) to the process $(e^{-rt}u(X^{x,\xi,\eta}_t))_{t\geq 0}$ on an arbitrary time interval $[0,T]$, $T > 0$. Then, recalling that $\xi^c$ and $\eta^c$ denote the continuous parts of $\xi$ and $\eta$, respectively, we have
\begin{align}
\label{verifico1}
u(x) = \ & e^{-r T} u(X^{x;\xi,\eta}_T) - \int_0^T e^{-rs} (\cL_X - r)u(X^{x;\xi,\eta}_s)\ ds - M^{x;\xi,\eta}_T \nonumber \\
&  - \int_0^T  e^{-rs} u'(X^{x;\xi,\eta}_s)\ d\xi_s^c + \int_0^T  e^{-rs} u'(X^{x;\xi,\eta}_s)\ d\eta_s^c \\
& - \sum_{0 \leq s < T}  e^{-rs} \Big(u(X^{x;\xi,\eta}_{s+}) - u(X^{x;\xi,\eta}_s)\Big), \nonumber
\end{align}
where we have set
$$ M^{x;\xi,\eta}_T := \int_0^T  e^{-rs} \sigma(X^{x;\xi,\eta}_s) u'(X^{x;\xi,\eta}_s)\ dB_s. $$
Since the processes $\xi$ and $\eta$ jump on disjoint subsets of $\R_+$ we can write
$$  \sum_{0 \leq s < T}  e^{-rs} (u(X_{s+}) - u(X_s)) =  \sum_{0 \leq s < T}  e^{-rs} \bigg[ \int_0^{\Delta \xi_s} u'(X^{x;\xi,\eta}_s + z)\ dz -  \int_0^{\Delta \eta_s} u'(X^{x;\xi,\eta}_s - z)\ dz \bigg], $$
and because $(\cL_X - r) u \geq - h$ and $- c_1 \leq u' \leq c_2$ on $\cI$ by \eqref{HJB}, we end up from \eqref{verifico1} with
\begin{align}
\label{verifico1bis}
u(x) \leq \ & e^{-r T} u(X^{x;\xi,\eta}_T) + \int_0^T e^{-rs} h(X^{x;\xi,\eta}_s)\ ds - M^{x;\xi,\eta}_T \nonumber \\
& + \int_0^T  e^{-rs} c_1(X^{x;\xi,\eta}_s)\,{\scriptstyle{\oplus}}\,d\xi_s + \int_0^T  e^{-rs} c_2(X^{x;\xi,\eta}_s)\,{\scriptstyle{\ominus}}\,d\eta_s,
\end{align}
upon recalling \eqref{defintegral1} and \eqref{defintegral2}. 

By assumption, for all $x \in \cI$ one has  $|u(x)| \leq K ( 1 + |x|^{\gamma + 1})$, and therefore we can write for some $K>0$ that
\begin{align}
\label{verifico1tris}
u(x) \leq \ & e^{-\frac{r}{2} T}  K \Big( 1 + |X^{x;\xi,\eta}_T|^{\gamma + 1}\Big)e^{- \frac{r}{2} T}  + \int_0^T e^{-rs} h(X^{x;\xi,\eta}_s)\ ds - M^{x;\xi,\eta}_T  \nonumber \\
&  + \int_0^T  e^{-rs} c_1(X^{x;\xi,\eta}_s) \,{\scriptstyle{\oplus}}\, d\xi_s + \int_0^T  e^{-rs} c_2(X^{x;\xi,\eta}_s) \,{\scriptstyle{\ominus}}\, d\eta_s  \nonumber \\
\leq \ & e^{-\frac{r}{2} T}  K \Big( 1 + \sup_{t \geq 0} e^{- \frac{r}{2} t} |X^{x;\xi,\eta}_t|^{\gamma + 1} \Big)  + \int_0^T e^{-rs} h(X^{x;\xi,\eta}_s)\ ds - M^{x;\xi,\eta}_T  \\
&  + \int_0^T  e^{-rs} c_1(X^{x;\xi,\eta}_s) \,{\scriptstyle{\oplus}}\, d\xi_s + \int_0^T  e^{-rs} c_2(X^{x;\xi,\eta}_s) \,{\scriptstyle{\ominus}}\, d\eta_s. \nonumber
\end{align}
From the previous equation we have that, for all $T > 0$, 
\begin{align*}
M^{x;\xi,\eta}_T  \leq & - u(x) + K \Big( 1 + \sup_{t \geq 0} e^{- \frac{r}{2} t} |X^{x;\xi,\eta}_t|^{\gamma + 1} \Big)  + \\
&    + \int_0^{\infty}  e^{-rs} |c_1(X^{x;\xi,\eta}_s)| \,{\scriptstyle{\oplus}}\, d\xi_s + \int_0^\infty  e^{-rs} |c_2(X^{x;\xi,\eta}_s)| \,{\scriptstyle{\ominus}}\, d\eta_s  
\end{align*}
so that $M^{x;\xi,\eta}_T \in L^1(\P)$ by admissibility of $\nu$ (cf.\ Definition \eqref{def:admiss}); hence, $(M^{x;\xi,\eta}_T)_{T\geq 0}$ is a submartingale. Then, taking expectations in \eqref{verifico1tris} we have
\begin{align*}
u(x) \leq \ & e^{-\frac{r}{2} T} \E_x\Big[K \Big(1 + \sup_{t \geq 0} e^{- \frac{r}{2} t} |X^{x;\xi,\eta}_t|^{\gamma + 1} \Big)\Big]  \nonumber \\
& + \E_x\bigg[\int_0^T e^{-rs} h(X^{x;\xi,\eta}_s)\ ds + \int_0^T  e^{-rs} \Big(c_1(X^{x;\xi,\eta}_s) \,{\scriptstyle{\oplus}}\, d\xi_s + c_2(X^{x;\xi,\eta}_s) \,{\scriptstyle{\ominus}}\, d\eta_s\Big)  \bigg] \nonumber
\end{align*}
Taking limits as $T \uparrow + \infty$, and using the fact that $\nu$ is admissible (cf.\ Definition \ref{def:admiss}), by the dominated convergence theorem we get $u(x) \leq \mathcal{J}_x(\nu)$. Since the latter holds for any $x \in \cI$ and $\nu \in \A(x)$ we conclude that $u \leq v$ on $\cI$. 
\vspace{0.25cm}

\emph{Step 2.} Let again $x \in \cI$ be given and fixed, and take the admissible $\widehat{\nu}$ satisfying \eqref{insideC}, \eqref{mg} and \eqref{flatoff}. Then all the inequalities leading to \eqref{verifico1bis} become equalities, and taking expectations we obtain
\begin{align}
\label{verifico3}
u(x) = \ & \E_x\bigg[e^{-r T} u(X^{x;\widehat{\xi},\widehat{\eta}}_T)  + \int_0^T e^{-rs} h(X^{x;\widehat{\xi},\widehat{\eta}}_s)\ ds +  \\
&   + \int_0^T  e^{-rs} c_1(X^{x;\widehat{\xi},\widehat{\eta}}_s) \,{\scriptstyle{\oplus}}\, d\widehat{\xi}_s + \int_0^T  e^{-rs} c_2(X^{x;\widehat{\xi},\widehat{\eta}}_s) \,{\scriptstyle{\ominus}}\, d\widehat{\eta}_s  \bigg]. \nonumber 
\end{align}

By assumption we have that, for any $x \in \cI$, $u(x) \geq - K(1 + |x|^{1+\gamma})$, so that we can continue from \eqref{verifico3} by writing
\begin{align}
\label{verifico4}
u(x) \geq & - e^{-\frac{r}{2} T}\E_x\Big[ K \Big(1 + \sup_{t \geq 0} e^{- \frac{r}{2} t} |X^{x;\widehat{\xi},\widehat{\eta}}_t|^{\gamma + 1} \Big)\Big] + \E_x\bigg[\int_0^T e^{-rs} h(X^{x;\widehat{\xi},\widehat{\eta}}_s)\ ds +  \\
&   + \int_0^T  e^{-rs} c_1(X^{x;\widehat{\xi},\widehat{\eta}}_s) \,{\scriptstyle{\oplus}}\, d\widehat{\xi}_s + \int_0^T  e^{-rs} c_2(X^{x;\widehat{\xi},\widehat{\eta}}_s) \,{\scriptstyle{\ominus}}\, d\widehat{\eta}_s  \bigg]. \nonumber 
\end{align}
By admissibility of $\widehat{\nu}$ (see Definition \ref{def:admiss}) we can take limits as $T \uparrow \infty$, invoke the dominated convergence theorem for the second expectation in the right hand-side of \eqref{verifico4}, and finally find that 
\begin{equation*}
u(x) \geq  \E_x\bigg[\int_0^{\infty} e^{-rs} h(X^{x;\widehat{\xi},\widehat{\eta}}_s)\ ds + \int_0^{\infty}  e^{-rs} c_1(X^{x;\widehat{\xi},\widehat{\eta}}_s) \,{\scriptstyle{\oplus}}\, d\widehat{\xi}_s + \int_0^{\infty}  e^{-rs} c_2(X^{x;\widehat{\xi},\widehat{\eta}}_s) \,{\scriptstyle{\ominus}}\, d\widehat{\eta}_s  \bigg]. 
\end{equation*}
Hence $u(x) \geq \mathcal{J}_x(\widehat{\nu}) \geq v(x)$. Combining this inequality with the fact that $u \leq v$ on $\cI$ by \emph{Step 1}, we conclude that $u=v$ on $\cI$ and that $\widehat{\nu}$ is optimal.
\end{proof}


\subsection{Constructing a Candidate Solution}
\label{sec:construct}

We here construct a solution to the HJB equation \eqref{HJB}. In particular, given the structure of our problem, we conjecture that there exist two constant trigger values to be determined, say $a$ and $b$, such that 
$$\big\{x\in\cI: \big(\cL_X-r\big)u(x) + h(x) =0\big\} = (a,b),$$
and that 
\begin{equation}
\label{A1A2}
\big\{x\in\cI: u'(x) = -c_1(x) \big\} = (\underline{x},a], \quad \text{and} \quad \big\{x\in\cI: u'(x) = c_2(x) \big\} = [b,\overline{x}).
\end{equation}

Following this conjecture we thus start by solving the ODE 
\begin{equation} 
\label{ODE-continuation}
(\cL_X-r)u(x) + h(x) = 0
\end{equation}
in $(a,b) \subset \cI$, for some $a<b$ to be found. Recalling \eqref{resolventsODE}, the general solution to equation \eqref{ODE-continuation} is given by
\begin{equation}
\label{ucont}
u(x) = A \psi(x) + B \phi(x) + (R h)(x), \quad x \in (a,b),
\end{equation}
for some $A,B \in \R$. Also, with regard to \eqref{A1A2} we set 
$$u(x) = A\psi(a) + B\phi(a)  + (R h)(a) + \int_x^a c_1(y) dy$$ 
for any $x \in (\underline{x},a]$, and 
$$u(x) = A\psi(b) + B\phi(b)  + (R h)(b) + \int_b^x c_2(y) dy$$ 
for any $x \in [b,\overline{x})$. Notice that in this way the function $u$ is automatically continuous at $a$ and $b$.

In order to determine the four unknown constants $A$, $B$, $a$, and $b$, we assume that $u \in C^2(\cI)$, and we then find the nonlinear system of four equations
\begin{align}
A \psi'(a) + B \phi'(a) + (R h)'(a) & = (\widehat R \, \widehat c_1)(a), \label{1}\\
A \psi''(a) + B \phi''(a) + (R h)''(a) & =  (\widehat R \, \widehat c_1)'(a), \label{2}\\
A \psi'(b) + B \phi'(b) + (R  h)'(b) & =  -(\widehat R \, \widehat c_2)(b), \label{3}\\
A \psi''(b) + B \phi''(b) + (R h)''(b) & =  -(\widehat R \, \widehat c_2)'(b). \label{4}
\end{align}

Solving \eqref{1}--\eqref{2} with respect to $A$ and $B$ and using \eqref{relation-resolvents}, simple but tedious algebra and the fact that $\psi' = \widehat \psi$ and $\phi' = - \widehat \phi$ (cf.\ Lemma \ref{lem:AM} in the appendix) give
\begin{align}
A & =  \frac{\widehat \phi'(a) [\widehat R (h' - \widehat c_1)](a) - \widehat \phi(a) [\widehat R (h' - \widehat c_1)]'(a)}{\widehat \phi(a) \widehat \psi'(a) - \widehat \phi'(a) \widehat \psi(a)}, \label{Aa} \\
B & =  \frac{\widehat \psi'(a) [\widehat R (h' - \widehat c_1)](a) - \widehat \psi(a) [\widehat R (h' - \widehat c_1)]'(a)}{\widehat \phi(a) \widehat \psi'(a) - \widehat \phi'(a) \widehat \psi(a)}. \label{Ba}
\end{align}

Analogous calculations starting from \eqref{3}--\eqref{4} reveal
\begin{align*}
A & =  \frac{\widehat \phi'(b) [\widehat R (h' + \widehat c_2)](b) - \widehat \phi(b) [\widehat R (h' + \widehat c_2)]'(b)}{\widehat \phi(b) \widehat \psi'(b) - \widehat \phi'(b) \widehat \psi(b)}, \\
B & = \frac{\widehat \psi'(b) [\widehat R (h' + \widehat c_2)](b) - \widehat \psi(b) [\widehat R (h' + \widehat c_2)]'(b)}{\widehat \phi(b) \widehat \psi'(b) - \widehat \phi'(b) \widehat \psi(b)}. 
\end{align*}

Recalling \eqref{Wronskian}, we can then write $A = I_1(a) = I_2(b)$ and $B = J_1(a) = J_2(b)$, with
\begin{align*}
I_1(a) & :=  \frac{1}{w} \left[ \frac{\widehat \phi'(a)}{\widehat S'(a)} [\widehat R(h' - \widehat c_1)](a) - \frac{\widehat \phi(a)}{\S'(a)} [\widehat R(h' - \widehat c_1)]'(a) \right], \\
I_2(b) & :=  \frac{1}{w} \left[ \frac{\widehat \phi'(b)}{\widehat S'(b)} [\widehat R(h' + \widehat c_2)](b) - \frac{\widehat \phi(b)}{\S'(b)} [\widehat R(h' + \widehat c_2)]'(b) \right], \\
J_1(a) & :=  \frac{1}{w} \left[ \frac{\widehat \psi'(a)}{\widehat S'(a)} [\widehat R(h' - \widehat c_1)](a) - \frac{\widehat \psi(a)}{\S'(a)} [\widehat R(h' - \widehat c_1)]'(a) \right], \\
J_2(b) & :=  \frac{1}{w} \left[ \frac{\widehat \psi'(b)}{\widehat S'(b)} [\widehat R(h' + \widehat c_2)](b) - \frac{\widehat \psi(b)}{\S'(b)} [\widehat R(h' + \widehat c_2)]'(b) \right],
\end{align*}
so that the system for $a$ and $b$ reads
$$ I_1(a) - I_2(b) = 0, \qquad J_1(a) - J_2(b) = 0.$$

We now make the following standing assumption.
\begin{assum} 
\label{A3}
One has that
$$ \lim_{x \to \underline x} J_i(x) = 0 = \lim_{x \to \overline x} I_i(x), \quad i = 1,2.$$
\end{assum}
\noindent By \eqref{psiphiproperties1}--\eqref{psiphiproperties2bis}, the latter is essentially a requirement on the growth of $\widehat R(h' - \widehat c_1)$ and $\widehat R(h' + \widehat c_2)$, and of their derivatives. Assumption \ref{A3} then implies that for any $x \in \cI$

\begin{equation}
\label{rewriting}
I_i(x) = - \int_x^{\overline x} I'_i(z)\ dz,  \qquad J_i(x) =  \int_{\underline x}^x J'_i(z)\ dz, \qquad i = 1,2.
\end{equation}

Notice now that for any function $f \in C^2(\cI)$, standard differentiation, and the fact that $\cL_{\X}\S =0$ and $(\cL_{\X}-(r-\mu'))g = 0$ for $g=\widehat{\psi},\widehat{\phi}$, yield
\begin{equation}
\label{derivative1}
\frac{d}{dx} \left[ \frac{ f'(x)}{\S'(x)} \widehat \phi(x) - \frac{\widehat \phi'(x)}{\S'(x)} f(x) \right] = \widehat \phi(x) \widehat m'(x) (\cL_{\X} - (r - \mu'(x))) f(x), 
\end{equation}
and
\begin{equation}
\label{derivative2}
\frac{d}{dx} \left[  \frac{ f'(x)}{\S'(x)} \widehat \psi(x) - \frac{\widehat \psi'(x)}{\S'(x)} f(x) \right] = \widehat \psi(x) \widehat m'(x) (\cL_{\X} - (r - \mu'(x))) f(x).
\end{equation}
As a consequence, using \eqref{rewriting} we have that $I_1(a) = I_2(b)$ is equivalent to 
\begin{equation}
\label{system-a}
\int_a^{\overline x} \big(h'(z) - \widehat c_1(z)\big) \widehat m'(z) \widehat \phi(z)\ dz =  \int_b^{\overline x} \big(h'(z) + \widehat c_2(z)\big) \widehat m'(z) \widehat \phi(z)\ dz, 
\end{equation}
whereas $J_1(a) = J_2(b)$ is equivalent to
\begin{equation}
\label{system-b}
\int_{\underline x}^a \big(h'(z) - \widehat c_1(z)\big) \widehat m'(z) \widehat \psi(z)\ dz =  \int_{\underline x}^b \big(h'(z) + \widehat c_2(z)\big) \widehat m'(z) \widehat \psi(z)\ dz. 
\end{equation}

Since we are looking for a solution $(a^*,b^*)$ of \eqref{system-a} and \eqref{system-b} such that $a^*<b^*$, we can rewrite them in the form
\begin{align}
\int_a^b  \big(h'(z) - \widehat c_1(z)\big) \widehat m'(z) \widehat \phi(z)\ dz & =  \int_b^{\overline x}  \big(\widehat c_1(z) + \widehat c_2(z)\big) \widehat m'(z) \widehat \phi(z)\ dz, \label{first} \\
\int_a^b  \big(h'(z) + \widehat c_2(z)\big) \widehat m'(z) \widehat \psi(z)\ dz & =  - \int_{\underline x}^a  \big(\widehat c_1(z) + \widehat c_2(z)\big) \widehat m'(z) \widehat \psi(z)\ dz. \label{second}
\end{align}

\begin{prop}
\label{prop:existenceab}
Recall $\widetilde x_1$, $\widetilde x_2$ as in Assumption \ref{A2}-(ii). Then there exists a unique couple $(a^*,b^*) \in \mathcal{I}\times \mathcal{I}$ such that $a^* < \widetilde{x}_1 < \widetilde{x}_2 < b^*$ that solves the system of equations \eqref{first} and \eqref{second}.
\end{prop}
\begin{proof}
\emph{Step 1.} We start by proving existence. Given Assumption \ref{A2}, note that the right-hand sides of \eqref{first} and \eqref{second} are strictly negative and strictly positive, respectively. For $a,b \in \cI$ define the two functionals
\begin{align*}
K_1(a;b) & :=   \int_a^b \big(h'(z) - \widehat c_1(z)\big) \widehat m'(z) \widehat \phi(z)\ dz, \\
K_2(b;a) & :=   \int_a^b \big(h'(z) + \widehat c_2(z)\big) \widehat m'(z) \widehat \psi(z)\ dz.
\end{align*}
For a given and fixed $a \in \cI$, let $b > a \lor \widetilde{x}_2$ and notice that by the integral mean-value theorem there exists $\xi_2 \in (a \lor \widetilde{x}_2,y)$ such that 
\begin{align}
\label{stimaK2}
K_2(b;a) & = K_2(a \lor \widetilde{x}_2;a) + \int_{a \lor \widetilde{x}_2}^b  \big(h'(z) + \widehat c_2(z)\big) \widehat m'(z) \widehat \psi(z)\ dz \nonumber \\
& = K_2(a \lor \widetilde{x}_2;a) +  \big(h'(\xi_2) + \widehat c_2(\xi_2)\big) \int_{a \lor \widetilde{x}_2}^b \frac{r - \mu'(z)}{r - \mu'(z)}  \widehat m'(z) \widehat \psi(z)\ dz  \\
& \geq K_2(a \lor \widetilde{x}_2;a) + \big(h'(\xi_2) + \widehat c_2(\xi_2)\big) \cdot \frac{1}{r + L} \left[ \frac{\widehat \psi'(b)}{\widehat S'(b)} - \frac{\widehat \psi'(a \lor \widetilde{x}_2)}{\widehat S'(a \lor \widetilde{x}_2)} \right],\nonumber
\end{align}
where in the last step we have used \eqref{psiphiproperties3}, the fact that $h'(\xi_2) + \widehat c_2(\xi_2) > 0$, as well as that $- L \leq \mu'(\,\cdot\,) < r$ by Assumptions \ref{A1} and \ref{ass:rate}. Because of \eqref{psiphiproperties2bis}, and again since $h'(\xi_2) + \widehat c_2(\xi_2) > 0$, we obtain from \eqref{stimaK2} that $\lim_{b \uparrow \overline x}K_2(b;a) = + \infty$, for any given $a \in \cI$. 

On the other hand, by Assumption \ref{A2} one has
\begin{align*}
K_2(\widetilde{x}_2;a) & = \int_a^{\widetilde{x}_2} \big(h'(z) + \widehat c_2(z)\big) \widehat m'(z) \widehat \psi(z)\ dz \left\{ \begin{array}{ll}
< 0 & \mbox{ if } a < \widetilde{x}_2, \\
\leq 0 & \mbox{ if } a \geq \widetilde{x}_2.\\
\end{array} \right. 
\end{align*}
Also, $K_2(a;a)=0$ and $K'_2(b;a) = \big(h'(b) + \widehat c_2(b)\big) \widehat m'(b) \widehat \psi(b) > 0$ for $b> a \lor \widetilde{x}_2$. Hence, for any given $a \in \cI$, by continuity and strict monotonicity of $b \mapsto K_2(b;a)$ on $(a \lor \widetilde{x}_2,\overline x)$, there exists a unique $y^*(a) \in (a \lor \widetilde{x}_2,\overline x)$ such that \eqref{second} is satisfied. 

Analogously, for fixed $b\in \cI$, take $a < \widetilde x_1 \land b$, and for a suitable $\xi_1 \in (a, \widetilde x_1 \land b)$ one finds
\begin{align}
\label{stimaK1}
K_1(a;b) & \leq K_1(\widetilde{x}_1 \land b;b) +  \big(h'(\xi_1) - \widehat c_1(\xi_1)\big) \cdot \frac{1}{r + L} \left[\frac{\widehat \phi'(\widetilde{x}_1 \land b)}{\widehat S'(\widetilde{x}_1 \land b)} - \frac{\widehat \phi'(a)}{\widehat S'(a)} \right].
\end{align}
We thus conclude that, for any given and fixed $b \in \cI$, $\lim_{a \downarrow \underline x}K_1(a;b) = - \infty$, since $\underline x$ is natural for $\widehat X$ (cf.\ \eqref{psiphiproperties2bis}) and $h'(\xi_1) - \widehat c_1(\xi_1)<0$. On the other hand, $K_1(b;b) = 0$,
$$ K_1(\widetilde{x}_1;b) = \int_{\widetilde{x}_1}^b \big(h'(z) - \widehat c_1(z)\big) \widehat m'(z) \widehat \phi(z)\ dz = \left\{ \begin{array}{ll}
> 0 & \mbox{ if } b > \widetilde{x}_1, \\
\geq 0 & \mbox{ if } b \leq \widetilde{x}_1,\\
\end{array} \right. $$
and $K'_1(a;b)=-\big(h'(a) + \widehat c_1(a)\big) \widehat m'(a) \widehat \phi(a) < 0$ for $a < \widetilde{x}_1 \land b$. Combining all these facts we find that for any $b \in \cI$ there exists a unique $x^*(b) \in (\underline x, \widetilde{x}_1 \land b)$ such that \eqref{first} is satisfied. Since $\widetilde{x}_1 < \widetilde{x}_2$ by assumption, we clearly have that if a pair $(a^*,b^*) \in \cI \times \cI$ such that $a^* := x^*(b^*)$ and $b^* = y^*(a^*)$ exists, then $a^* < \widetilde{x}_1 < \widetilde{x}_2 < b^*$. 


In order to prove that there indeed exists such a couple $(a^*,b^*)$, let
$$ \Theta(b) := \int_{x^*(b)}^b  \big(h'(z) + \widehat c_2(z)\big) \widehat m'(z) \widehat \psi(z)\ dz + \int_{\underline x}^{x^*(b)} \big(\widehat c_1(z) + \widehat c_2(z)\big) \widehat m'(z) \widehat \psi(z)\ dz, $$
and notice that because $x^*(\widetilde{x}_2) < \widetilde{x}_1< \widetilde{x}_2$ one has by Assumption \ref{A2}
$$ \Theta(\widetilde{x}_2) =  \int_{x^*( \widetilde{x}_2)}^{ \widetilde{x}_2}  \big(h'(z) + \widehat c_2(z)\big) \widehat m'(z) \widehat \psi(z)\ dz + \int_{\underline x}^{x^*( \widetilde{x}_2)} \big(\widehat c_1(z) + \widehat c_2(z)\big) \widehat m'(z) \widehat \psi(z)\ dz < 0. $$

On the other hand, for $b > \widetilde{x}_2$ and for a suitable $\xi_2 \in (\widetilde{x}_2,b)$ we have by the integral mean-value theorem
\begin{align*}
\Theta(b) > & \int_{x^*(b)}^{ \widetilde{x}_2}  \big(h'(z) + \widehat c_2(z)\big) \widehat m'(z) \widehat \psi(z)\ dz + \int_{ \widetilde{x}_2}^b \big(h'(z) + \widehat c_2(z)\big) \widehat m'(z) \widehat \psi(z)\ dz  \\
&   + \int_{\underline x}^{\widetilde{x}_2} \big(\widehat c_1(z) + \widehat c_2(z)\big) \widehat m'(z) \widehat \psi(z)\ dz  \\
\geq &  \int_{\underline x}^{ \widetilde{x}_2}  \big(h'(z) + \widehat c_2(z)\big) \widehat m'(z) \widehat \psi(z)\ dz + \int_{\underline x}^{\widetilde{x}_2} \big(\widehat c_1(z) + \widehat c_2(z)\big) \widehat m'(z) \widehat \psi(z)\ dz   \\
&    + \big(h'(\xi_2) + \widehat c_2(\xi_2)\big) \int_{ \widetilde{x}_2}^b \widehat m'(z) \frac{r - \mu'(z)}{r - \mu'(z)} \widehat \psi(z)\ dz  \\
\geq &  \int_{\underline x}^{ \widetilde{x}_2}  \big(h'(z) + \widehat c_2(z)\big) \widehat m'(z) \widehat \psi(z)\ dz + \int_{\underline x}^{\widetilde{x}_2} \big(\widehat c_1(z) + \widehat c_2(z)\big) \widehat m'(z) \widehat \psi(z)\ dz   \\
&  + \big(h'(\xi_2) + \widehat c_2(\xi_2)\big) \frac{1}{r + L} \left[\frac{\widehat \psi'(b)}{\widehat S'(b)} - \frac{\widehat \psi'( \widetilde{x}_2)}{\widehat S'( \widetilde{x}_2)} \right]. 
\end{align*}
Here we have used that $\widehat{c}_1 + \widehat{c}_2 < 0$ on $\cI$ by Assumption \ref{A2}, equation \eqref{psiphiproperties3}, as well as that $- L \leq \mu'(\,\cdot\,) < r$ by Assumptions \ref{A1} and \ref{ass:rate}. Because of \eqref{psiphiproperties2bis}, and since $h'(\xi_2) + \widehat c_2(\xi_2)>0$, we find from the previous equation that $\lim_{b \uparrow \overline x}\Theta(b) = + \infty$. Also, $b \mapsto \Theta(b)$ is continuous (since $b \mapsto x^*(b)$ is so by the implicit function theorem), and therefore there exists $b^*\in \cI$ solving $\Theta(b) = 0$, and this value is such that $b^* > \widetilde{x}_2$. Hence, there also exists $a^* = x^*(b^*)$ such that $(a^*,b^*) \in (\underline x, \widetilde{x}_1) \times (\widetilde{x}_2,\overline x)$ solves system \eqref{first}--\eqref{second}. 
\vspace{0.25cm}

\emph{Step 2.} We now prove that the couple $(a^*,b^*)$ is indeed unique in the domain $(\underline x, \widetilde{x}_1)\times(\widetilde{x}_2,\overline x)$. For $x^*(b)$, $b \in \cI$, and $y^*(a)$, $a \in \cI$, the implicit function theorem gives
\begin{align*}
(x^*)'(b) & =  \frac{h'(b) + \widehat c_2(b)}{h'(x^*(b)) - \widehat c_1(x^*(b))} \cdot \frac{\widehat m'(b) \widehat \phi(b)}{\widehat m'(x^*(b)) \widehat \phi(x^*(b))}, \\
(y^*)'(a) & =  \frac{h'(a) - \widehat c_1(a)}{h'(y^*(a)) + \widehat c_2(y^*(a))} \cdot \frac{\widehat m'(a) \widehat \psi(a)}{\widehat m'(y^*(a) \widehat \psi(y^*(a))}.
\end{align*}
Since we already know by \emph{Step 1} that there exists a solution $(a^*,b^*)$ belonging to $(\underline x, \widetilde{x}_1) \times (\widetilde{x}_2,\overline x)$, we can study the sign of the previous derivatives in those intervals. By Assumption \ref{A2} we have
$$ (x^*)'(b) < 0 \mbox{ for } b \in (\widetilde{x}_2,\overline x) \qquad \mbox{ and } \qquad (y^*)'(a) < 0  \mbox{ for } a \in (\underline x, \widetilde{x}_1), $$
upon recalling that $x^*(b) < \widetilde{x}_1$ and $y^*(a) > \widetilde{x}_2$. Moreover, at the intersection points $b^* = y^*(a^*)$ and $a^* = x^*(b^*)$, we have
\begin{align*}
(x^*)'(b^*) & =  \frac{h'(b^*) + \widehat c_2(b^*)}{h'(a^*) - \widehat c_1(a^*)} \cdot \frac{\widehat m'(b^*) \widehat \psi(b^*)}{\widehat m'(a^*) \widehat \psi(a^*)} \cdot \frac{\widehat \psi(a^*)}{\widehat \psi(b^*)} \cdot \frac{\widehat \phi(b^*)}{\widehat \phi(a^*)}  \\
& =  \frac{1}{(y^*)'(a^*)} \cdot \frac{\widehat \psi(a^*)}{\widehat \psi(b^*)} \cdot \frac{\widehat \phi(b^*)}{\widehat \phi(a^*)}.
\end{align*}
Since $(y^*)'(a^*) < 0$ and $\frac{\widehat \psi(a^*)}{\widehat \psi(b^*)} \cdot \frac{\widehat \phi(b^*)}{\widehat \phi(a^*)} < 1$ (by the strict monotonicity of $\widehat \psi$ and $\widehat \phi$), we obtain that 
$(x^*)'(b^*) >  \frac{1}{(y^*)'(a^*)}$, or equivalently
$$ (y^*)'(a^*) >  \frac{1}{(x^*)'(b^*)} = [(x^*)^{-1}]^{'}(a^*). $$
Together with the strict monotonicity of $x^*(\,\cdot\,)$ and $y^*(\,\cdot\,)$ in  $(\underline x, \widetilde{x}_1)$ and $(\widetilde{x}_2,\overline x)$, respectively, the latter shows that the intersection point is indeed unique. 
\end{proof}

Now, with $(A,B,a^*,b^*)$ as above, we define our candidate value function as
\begin{equation}
\label{u-candidate}
u(x) := \left\{ \begin{array}{ll}
\displaystyle A \psi(a^*) + B \phi(a^*) + (R h)(a^*) + \int_x^{a^*} c_1(y)\ dy, & x \in (\underline x, a^*], \\
\\
A \psi(x) + B \phi(x) + (R h)(x), & x \in (a^*, b^*), \\
\\
\displaystyle A \psi(b^*) + B \phi(b^*) + (R h)(b^*) + \int_{b^*}^x c_2(y)\ dy, & x \in [b^*,\overline x). \\
\end{array} \right. 
\end{equation}


\subsection{The Value Function and the Optimal Control}
\label{sec:optimal}

In this section we prove that the function $u$ constructed in Section \ref{sec:construct} (cf.\ \eqref{u-candidate} above) coincides with the value function \eqref{value}, and we provide the optimal control ${\nu}^{\star}$.

\begin{thm}
\label{uHJB}
The function $u$ defined in \eqref{u-candidate} is a classical solution to the HJB equation \eqref{HJB}. Moreover, there exists $K>0$ such that $|u(x)| \leq K\big(1 + |x|^{\gamma + 1}\big)$, where $\gamma \geq 1$ is the growth coefficient of $c_i$, $i=1,2$ (see Assumption \ref{A2}-(ii)). 
\end{thm}

\begin{proof}
The proof is organized in several steps.
\vspace{0.25cm}

\emph{Step 1.} By construction, $u \in C^2(\underline x, \overline x)$. Moreover, using the growth requirement on $c_i$, $i=1,2$, of Assumption \ref{A2}-(ii), one obtains from \eqref{u-candidate} that there exists $K>0$ such that for all $x \in \cI$
$$|u(x)| \leq K\bigg(1 + \int_{x \land a^*}^{a^*} |c_1(y)|\ dy + \int_{b^*}^{x \lor b^*} |c_2(y)|\ dy \bigg) \leq K \big( 1 + |x|^{\gamma + 1}\big).$$
\vspace{0.25cm}

\emph{Step 2.} We here show that $\big(\cL_X - r\big)u(x) + h(x) \geq 0$ for all $x \in \cI$. This is clearly true with equality by construction for $x \in (a^*,b^*)$, and we now prove that it also holds for $x \in (\underline{x}, a^*]$. Analogous arguments might then be employed also to show that $\big(\cL_X - r\big)u(x) + h(x) \geq 0$ for $x \in [b^*,\overline{x})$. 

First of all we rewrite $A \psi(a^*) + B \phi(a^*) + (Rh)(a^*)$ in a tighter form. 
Notice that since $- \widehat \phi = \phi'$ and $\widehat \psi = \psi'$ by Lemma \ref{lem:AM} in the appendix, and because $(\widehat R h') = (R h)'$, we can obtain from \eqref{Aa}--\eqref{Ba}
\begin{align*}
& {A \psi(a^*) + B \phi(a^*) = \frac{1}{\phi''(a^*) \psi'(a^*) - \phi'(a^*) \psi''(a^*)} \cdot} \\
& \cdot  \Big[ \psi(a^*) \phi'(a^*) (R h)''(a^*) - \psi(a^*) \phi'(a^*) (\widehat R\, \widehat c_1)'(a^*) - \psi(a^*) \phi''(a^*) (R h)'(a^*) + \\
&    + \psi(a^*) \phi''(a^*) (\widehat R\, \widehat c_1)(a^*) + \phi(a^*) \psi''(a^*) (R h)'(a^*) - \phi(a^*) \psi''(a^*) (\widehat R\, \widehat c_1)(a^*) - \\
&    - \phi(a^*) \psi'(a^*) (R_r h)''(a^*) + \phi(a^*) \psi'(a^*) (\widehat R\, \widehat c_1)'(a^*)\Big].
\end{align*}
Then arranging terms and noticing that $WS''(x) = \phi(x)\psi''(x) - \psi(x)\phi''(x)$, $x \in \cI$, we find from the previous equation that
\begin{align}
\label{arrange1}
& {A \psi(a^*) + B \phi(a^*) = \frac{W}{\phi''(a^*) \psi'(a^*) - \phi'(a^*) \psi''(a^*)} \cdot} \\
& \cdot  \Big[ -S'(a^*)\Big((R h)''(a^*) - (\widehat R\, \widehat c_1)'(a^*)\Big) + S''(a^*)\Big((R h)'(a^*) - (\widehat R\, \widehat c_1)(a^*)\Big) \Big]. \nonumber
\end{align}
Then using that $S''(x) = -\frac{2\mu(x)}{\sigma^2(x)}S'(x)$, and that $f''(x) = -\frac{2\mu(x)}{\sigma^2(x)}f'(x) + \frac{2r}{\sigma^2(x)}f(x)$ for $f=\phi,\psi$, we obtain from \eqref{arrange1}
\begin{align}
\label{arrange2}
& {A \psi(a^*) + B \phi(a^*) = - \frac{1}{r} \cdot} \\
& \cdot  \Big[ \frac12 \sigma^2(a^*) (R h){''}(a^*) + \mu(a^*)(R h){'}(a^*) - \Big( \frac12 \sigma^2(a^*)(\widehat R\, \widehat c_1)'(a^*) + \mu(a^*)(\widehat R\, \widehat c_1)(a^*)\Big) \Big]. \nonumber
\end{align}
Since the resolvent satisfies (cf.\ \eqref{resolventsODE})
$$ \frac12 \sigma^2(x) (R h){''}(x) + \mu(x)(R h){'}(x) = r (Rh)(x) - h(x), \quad x \in \cI,$$
and because $(\widehat R\, \widehat c_1) = - c_1$ by \eqref{repr-c}, we conclude by simple algebra from \eqref{arrange2} that
\begin{equation}
\label{arrange-final}
A \psi(a^*) + B \phi(a^*) + (R h)(a^*) = \frac{1}{r} \left[ h(a^*) - \mu(a^*) c_1(a^*) - \frac12 \sigma^2(a^*) c_1'(a^*) \right].
\end{equation}
Hence (cf.\ \eqref{u-candidate})
\begin{equation}
\label{nuovau}
u(x) = \int_x^{a^*} c_1(y)\ dy +  \frac{1}{r} \Big[ h(a^*) - \mu(a^*) c_1(a^*) - \frac12 \sigma^2(a^*) c_1'(a^*) \Big], \quad x \in (\underline x, a^*].
\end{equation}

Thanks to \eqref{nuovau} we can now easily check that $\big(\cL_X - r) u + h \geq 0$ on $(\underline{x}, a^*]$. Indeed, since $a^* < \widetilde{x}_1$, we have for any $x \leq a^*$ that
$$ -\int_x^{a^*} \Big(h'(y) - \big(\cL_{\X} - (r - \mu'(y))\big) c_1(y)\Big)\ dy \geq 0. $$
However, an integration by parts yields
\begin{align*}
0 \leq & \ - \int_x^{a^*} \Big(h'(y) - \big(\cL_{\X} - (r - \mu'(y))\big) c_1(y)\Big)\ dy \\
= & \ h(x) - h(a^*) + \frac12 \sigma^2(a^*) c_1'(a^*) -  \frac12 \sigma^2(x) c_1'(x)  \\
&   + \mu(a^*) c_1(a^*) - \mu(x) c_1(x) - r \int_x^{a^*} c_1(y)\ dy  \\
= & \ h(x) -  \frac12 \sigma^2(x) c_1'(x) - \mu(x) c_1(x) - r \int_x^{a^*} c_1(y)\ dy  \\
&  - r\Big[\frac{1}{r}\Big(h(a^*) - \frac12 \sigma^2(a^*) c_1'(a^*) - \mu(a^*) c_1(a^*) \Big)\Big]  \\
= & \ \big(\cL_X - r\big) u(x) + h(x)
\end{align*}
on $(\underline x,a^*]$, upon recalling \eqref{nuovau} in the last step. Hence, $\big(\cL_X - r\big) u(x) + h(x) \geq 0$ on $(\underline x,a^*]$.
\vspace{0.25cm}

\emph{Step 3.} To conclude the proof it remains to show that we have $- c_1 \leq u' \leq c_2$ on $(a^*,b^*)$, since we already know by construction that $u' = - c_1$ on $(\underline x,a^*]$ and $u' = c_2$ on $[b^*,\overline x)$. We here prove that $u' \geq -c_1$ on $(a^*,b^*)$. Arguments similar to those employed in the following allow to show that also $u' \leq c_2$ on $(a^*,b^*)$.

By construction, the function $u$ of \eqref{u-candidate} solves $(\cL_X - r)u = - h$ on $(a^*,b^*)$. Given the regularity of $u$, and of $\mu$ and $\sigma$ (cf.\ Assumption \ref{A1}), we can differentiate the previous equation with respect to $x$ inside $(a^*,b^*)$, and find that $u'$ solves 
$$ \big(\cL_{\X}-(r-\mu'(x))\big)u'(x) = - h'(x), \qquad x \in (a^*,b^*),$$
together with the boundary conditions $u'(a^*) = - c_1(a^*)$ and $u'(b^*)=c_2(b^*)$.

Now, take $x \in (a^*,b^*)$, and define the two stopping times
$$ \tau_{a^*} := \inf\{ t \geq 0\ |\ \widehat X_t \leq a^*\}, \qquad \tau_{b^*} := \inf\{ t \geq 0\ |\ \widehat X_t \geq b^*\},\quad  \widehat{\P}_x-\text{a.s.}$$
Then, by the Feynman-Kac formula and the strong Markov property of $\X$, we can write
\begin{align*}
u'(x) = & \widehat \E_x\bigg[ e^{- \int_0^{\tau_{a^*} \land \tau_{b^*}} (r - \mu'(\widehat X_u)) du} w(\widehat X_{\tau_{a^*} \land \tau_{b^*}}) + \int_0^{\tau_{a^*} \land \tau_{b^*}} e^{- \int_0^s (r - \mu'(\widehat X_u)) du}\, h'(\X_s)\ ds \bigg] \\
= & (\widehat R h')(x) - \widehat \E_x\Big[ e^{- \int_0^{\tau_{a^*}} (r - \mu'(\widehat X_u)) du} (\widehat R h' + c_1)(\widehat X_{\tau_{a^*}}) \1_{\{\tau_{a^*} < \tau_{b^*}\}} \Big]  \\
&    - \widehat \E_x\Big[ e^{- \int_0^{\tau_{b^*}} (r - \mu'(\widehat X_u)) du} (\widehat R h' - c_2)(\widehat X_{\tau_{b^*}}) \1_{\{\tau_{a^*} > \tau_{b^*}\}} \Big] \\
= & (\widehat R h')(x) -  (\widehat R h' + c_1)(a^*) \widehat \E_x\Big[ e^{- \int_0^{\tau_{a^*}} (r - \mu'(\widehat X_u)) du} \1_{\{\tau_{a^*} < \tau_{b^*}\}} \Big]  \\
&    - (\widehat R h' - c_2)(b^*)  \widehat \E_x\Big[ e^{- \int_0^{\tau_{b^*}} (r - \mu'(\widehat X_u)) du} \1_{\{\tau_{a^*} > \tau_{b^*}\}} \Big],
\end{align*}
upon recalling \eqref{resolventhat}. Because the function $\widehat F(x) := \frac{\widehat \psi(x)}{\widehat \phi(x)}$, $x \in \cI$, is strictly positive and strictly increasing (by strict monotonicity of $\widehat{\psi}$ and $\widehat{\phi}$), we can apply Lemma 2.3 in \cite{Dayanik08} to evaluate the last two expectations above, and then to write that   
\begin{align*}
\frac{u'(x)}{\widehat \phi(x)} & =  \frac{(\widehat R h')(x)}{\widehat \phi(x)} - \frac{(\widehat R h')(a^*) + c_1(a^*)}{\widehat \phi(a^*)} \cdot \left[\frac{\widehat F(b^*) - \widehat F(x)}{\widehat F(b^*) - \widehat F(a^*)}\right] \\
&  - \frac{(\widehat R h')(b^*) - c_2(b^*)}{\widehat \phi(b^*)} \cdot \left[\frac{\widehat F(x) - \widehat F(a^*)}{\widehat F(b^*) - \widehat F(a^*)}\right].
\end{align*}
The latter equation immediately implies that
\begin{align}
\label{uc1-1}
\widehat{w}_1(x):=\frac{u'(x) + c_1(x)}{\widehat \phi(x)} = & \frac{(\widehat R h')(x) + c_1(x)}{\widehat \phi(x)} - \frac{(\widehat R h')(a^*) + c_1(a^*)}{\widehat \phi(a^*)} \cdot \left[\frac{\widehat F(b^*) - \widehat F(x)}{\widehat F(b^*) - \widehat F(a^*)}\right]  \nonumber \\
&  - \frac{(\widehat R h')(b^*) - c_2(b^*)}{\widehat \phi(b^*)} \cdot \left[\frac{\widehat F(x) - \widehat F(a^*)}{\widehat F(b^*) - \widehat F(a^*)}\right].
\end{align}

We now want to show that $\widehat{w}_1 \geq 0$ on $(a^*,b^*)$ since this is clearly equivalent to proving that $u' \geq - c_1$ on such interval.
Clearly $\widehat{w}_1(a^*) = 0$, and by standard differentiation one also gets that $\widehat{w}'_1(a^*) = 0$ since $u{''}(a^*) + c'_1(a^*)=0$. Also, by \eqref{uc1-1} we have that
$\widehat{w}_1(b^*)=(c_1(b^*) + c_2(b^*))/\widehat{\phi}(b^*)>0$, where the last inequality is due to Assumption \ref{A2} and the positivity of $\widehat{\phi}$.

If we now can prove that $\widehat{w}'_1 \geq 0$ on $(a^*,b^*)$, then we have that $\widehat{w}_1(x) \geq \widehat{w}_1(a^*)=0$ for any $x \in (a^*,b^*)$, and therefore that $u' \geq - c_1$ on $(a^*,b^*)$. Recalling \eqref{repr-c}, we can rewrite \eqref{uc1-1} as
\begin{align}
\label{uc1-1-bis}
\widehat{w}_1(x) = & \frac{[\widehat R (h'-\widehat{c}_1)](x)}{\widehat \phi(x)} - \frac{[\widehat R (h'-\widehat{c}_1)](a^*)}{\widehat \phi(a^*)} \cdot \left[\frac{\widehat F(b^*) - \widehat F(x)}{\widehat F(b^*) - \widehat F(a^*)}\right]  \nonumber \\
&  - \frac{[\widehat R (h'+\widehat{c}_2)](x)}{\widehat \phi(b^*)} \cdot \left[\frac{\widehat F(x) - \widehat F(a^*)}{\widehat F(b^*) - \widehat F(a^*)}\right].
\end{align}
Then we can employ \eqref{resolventhatrepr} and \eqref{Greenhat} in \eqref{uc1-1-bis}, perform a standard differentiation with respect to $x$, and use the fact that $a^*$ and $b^*$ satisfy \eqref{system-a}--\eqref{system-b} to find after some algebra
\begin{align}
\label{uc1-2}
\widehat{w}'_1(x) & = - \frac{1}{w} \widehat{F}'(x) \int_{a^*}^{x}\big(h'(z) - \widehat{c}_1(z)\big)\widehat{\phi}(z)\widehat{m}'(z)\ dz =: - \frac{1}{w} \widehat{F}'(x) Q(x).
\end{align}

The function $Q$ introduced in \eqref{uc1-2} is such that $Q(a^*)=0$. Moreover, due to Assumption \ref{A2} we have that
\begin{equation}
\label{Q}
Q'(x) = \big(h'(x) - \widehat{c}_1(x)\big)\widehat{\phi}(x)\widehat{m}'(x) \left\{ \begin{array}{ll}
<0, & x < \widetilde{x}_1, \\
\\
=0, & x = \widetilde{x}_1, \\
\\
> 0, & x > \widetilde{x}_1, \\
\end{array} \right. 
\end{equation}
and $Q(x) < 0$ for any $x \in (a^*,\widetilde{x}_1]$. We now have two cases: either \emph{(i)} there exists $\widetilde{x}_1 < \ell_1 < b^*$ such that $Q(\ell_1)=0$ (and notice that, if it exists, such a point is unique by strict monotonicity of $Q(\,\cdot\,)$ on $(\widetilde{x}_1,\overline{x})$); or \emph{(ii)} $Q(x) \leq 0$ for any $x \in (a^*,b^*)$. 

In case \emph{(ii)}, we immediately conclude from \eqref{uc1-2} that $\widehat{w}'_1 \geq 0$ on $(a^*,b^*)$, and therefore that $u' \geq - c_1$ on $(a^*,b^*)$. 

On the other hand, if we are in case \emph{(i)}, by \eqref{uc1-2} we see that the point $\ell_1$ is also the unique stationary point of $\widehat{w}'_1$. In fact, it is a maximum of $\widehat{w}_1$ since one can easily derive from \eqref{uc1-2} that $\widehat{w}''_1(\ell_1) = -w^{-1}\widehat{F}'(\ell_1)Q'(\ell_1) < 0$, where the last inequality is due to the fact that $Q(\ell_1)=0$ but $Q'(\ell_1) > 0$. However, since we know that $\widehat{w}_1(a^*)=0$ and $\widehat{w}_1(b^*)>0$, we conclude that also in case \emph{(ii)} one has that $\widehat{w}'_1>0$ on $(a^*,b^*)$, and therefore that $u' \geq - c_1$ on $(a^*,b^*)$. 
\vspace{0.25cm}

\emph{Step 4.} Combining the results of the previous steps the proof is completed.
\end{proof}


Given $x\in \cI$, let $\nu^{\star}$ be such that $\nu^{\star} = \xi^{\star} - \eta^{\star}$ where $(\xi^{\star}, \eta^{\star})$ is the couple of nondecreasing processes that solves the following double Skorokhod reflection problem $\textbf{SP}(a^*,b^*;x)$
\begin{align}
\hspace{-10pt}
\text{Find $(\xi,\eta)\in \mathcal{U} \times \mathcal{U}$ s.t.}
\left\{
\begin{array}{l}
\displaystyle X^{x,\xi,\eta}_t\in[a^*,b^*], \text{$\P$-a.s.~for $t > 0$},\\[+10pt]
\displaystyle \int^{T}_0{\mathds{1}_{\{X^{x,\xi,\eta}_t>a^*\}}d\xi_t}=0, \text{$\P$-a.s.~for any $T>0$,}\\[+12pt]
\displaystyle \int^{T}_0{\mathds{1}_{\{X^{x,\xi,\eta}_t<b^*\}}d\eta_t}=0, \text{$\P$-a.s.~for any $T>0$.}
\end{array}
\right.
\end{align}
 
Under Assumption \ref{A1}, Problem $\textbf{SP}(a^*,b^*;x)$ admits a unique pathwise solution (cf., e.g., Theorem 4.1 in \cite{Tanaka}). Moreover, $t \mapsto \xi^{\star}_t$ and $t \mapsto \eta^{\star}_t$ are continuous, a part possible jumps at time zero of amplitude $(a^* - x)^+$ and $(x -b^*)^+$, respectively. It also follows that $\text{supp}\{d\xi^{\star}_t\}\cap \text{supp}\{d\eta^{\star}_t\}=\emptyset$. In the following, we set $\xi^{\star}_{0}=0=\eta^{\star}_{0}$ a.s., and, to simplify notation, $X^{\star}:=X^{\xi^{\star}, \eta^{\star}}$, $\mathbb{P}_x$-a.s.

\begin{prop}
Let $x \in \cI$ and let $(\xi^{\star}, \eta^{\star})$ solve $\textbf{SP}(a^*,b^*;x)$. Then the process $\nu^{\star} := \xi^{\star} - \eta^{\star}$ is an admissible control.
\end{prop}
\begin{proof}
Clearly $\nu^{\star} \in \mathcal{S}$. To prove the admissibility of $\nu^{\star}$, we have to verify the requirements of Definition \ref{def:admiss}. Since $X^{\star}_t\in[a^*,b^*] \subset \cI$ for all $t>0$, and $X^{\star}_0 = x \in \cI$, we have that $\mathcal{\sigma}_{\cI}=+\infty$ $\P_x$-a.s. Moreover, it is easy to see that also (b) and (c) of Definition \ref{def:admiss} are fulfilled.

It thus remains to show that Definition \ref{def:admiss}-(a) is satisfied as well.
By \eqref{defintegral1} and the fact that $(\xi^{\star}, \eta^{\star})$ solves $\textbf{SP}(a^*,b^*;x)$ we have
\begin{align}
\label{integraladm-2}
&\E_x\bigg[\int_{0}^{\infty} e^{-rt}|c_1(X^{\star}_t)| \circ d\xi^{\star}_t \bigg] = \int_0^{(a^*-x)^+}|c_1(x + z)| dz \nonumber \\
&+ |c_1(a^*)|\E_x\bigg[\int_{0}^{\infty} e^{-rt} d\xi^{c,\star}_t \bigg],
\end{align}
where we have used that $\{a^*\}$ is the support of the measure on $\mathbb{R}_+$, $d\xi^{c,\star}$, induced by the continuous part of $\xi^{\star}$. The continuity of $c_1$ yields
\begin{equation}
\label{inetgral1}
\int_0^{(a^*-x)^+}|c_1(x + z)| dz \leq (a^*-x)^+ \max_{u \in [0,(a^*-x)^+]}|c_1(x+u)| < \infty.
\end{equation}
Also, arguing as in the proof of Lemma 2.1 of \cite{Shreveetal} (see in particular equations (2.16)--(2.17) therein), we have that 
\begin{equation}
\label{inetgral2}
\E_x\bigg[\int_{0}^{\infty} e^{-rt} d\xi^{c,\star}_t \bigg] < \infty.
\end{equation}
Since analogous arguments apply to $\E_x[\int_{0}^{\infty} e^{-rt}|c_2(X^{\star}_t)| {\scriptstyle{\ominus}} d\eta^{\star}_t]$, by combining \eqref{inetgral1}, \eqref{inetgral2} and \eqref{integraladm-2}, we conclude that the requirement of Definition \ref{def:admiss}-(a) holds true, and therefore that $\nu^{\star}\in \mathcal{A}(x)$.
\end{proof}

\begin{thm}
\label{teo:verify}
Let $(\xi^{\star}, \eta^{\star})$ solving $\textbf{SP}(a^*,b^*;x)$, $\nu^{\star}$ such that $\nu^{\star} = \xi^{\star} - \eta^{\star}$, and $u$ as in \eqref{u-candidate}. Then one has that $u=v$ on $\cI$ and $\nu^{\star}$ is optimal for \eqref{value}.
\end{thm}
\begin{proof}
It suffices to check that the conditions of Theorem \ref{thm:verifico} are met. By Theorem \ref{uHJB} the function $u$ of \eqref{u-candidate} is a classical solution to the HJB equation \eqref{HJB}, and it is such that $|u(x)| \leq K\big(1 + |x|^{\gamma + 1}\big)$, for some $K>0$ and where $\gamma \geq 1$ is the growth coefficient of $c_i$, $i=1,2$ (see Assumption \ref{A2}-(ii)).

Moreover, $\nu^{\star} $ is such that $X^{\star}_t \in [a^*,b^*]=\big\{x \in \cI:\, \big(\cL_X - r\big)u(x) + h(x) =0\big\}$ for a.e.\ $t$ and a.s.; the process \eqref{mg} is an $\mathbb{F}$-martingale by continuity of $\sigma$ and $u'$; the conditions in \eqref{flatoff} are met due to the fact that $(\xi^{\star}, \eta^{\star})$ solves $\textbf{SP}(a^*,b^*;x)$, and that $\{x \in \cI: u'(x) = -c_1(x)\} = (\underline{x}, a^*]$ and $\{x \in \cI: u'(x) = c_2(x)\} = [b^*, \overline{x})$. Hence Theorem \ref{thm:verifico} applies and this completes the proof.
\end{proof}

\subsection{A Link with an Optimal Stopping Game}

The following proposition provides a probabilistic representation of the derivative $v'$ of the value function. This result plays an important role in the next section where we perform a comparative statics analysis.
\begin{prop}
\label{prop:OS}
Let $v$ be the value function of \eqref{value}. Then for any $x \in \cI$ one has
\begin{equation}
\label{gameOS}
v'(x) =  \inf_{\tau} \sup_{\sigma} \widehat{\mathcal{J}}_x(\tau,\sigma) = \sup_{\sigma} \inf_{\tau} \widehat{\mathcal{J}}_x(\tau,\sigma), 
\end{equation}
where $\tau,\sigma$, $i=1,2$ are $\widehat{\mathbb{F}}$-stopping times, and
\begin{align}
\label{functionalOS}
& \widehat{\mathcal{J}}_x(\tau,\sigma):=\widehat{\E}_x\bigg[ \int_0^{\tau \land \sigma} e^{- \int_0^s(r - \mu'(\X_u))du} h'(\widehat X_s)\ ds \nonumber \\
& - e^{- \int_0^{\sigma}(r - \mu'(\X_u))du}\, c_1(\X_{\sigma}) \1_{\{\sigma < \tau\}} + e^{- \int_0^{\tau}(r - \mu'(\X_u))du}\,c_2(\X_{\tau}) \1_{\{\tau < \sigma\}} \bigg].
\end{align}
Moreover, for any $x \in \cI$, the couple of $\widehat{\mathbb{F}}$-stopping times $(\tau^*,\sigma^*)$ given by
\begin{equation}
\label{saddle}
\tau^*:=\inf\{t\geq0: \X_t \geq b^*\} \qquad \mbox{and} \quad \sigma^*:=\inf\{t\geq0: \X_t \leq a^*\}, \quad \widehat{\mathbb{P}}_x-\mbox{a.s.}
\end{equation}
form a saddle-point; that is,  
$$\widehat{\mathcal{J}}_x(\tau^*,\sigma) \leq \widehat{\mathcal{J}}_x(\tau^*,\sigma^*) \leq \widehat{\mathcal{J}}_x(\tau,\sigma^*),$$
for any couple of $\widehat{\mathbb{F}}$-stopping times $(\tau,\sigma)$.
\end{prop}

\begin{proof}
We only provide a sketch of the proof, since its arguments are quite standard. From Theorem \ref{teo:verify} we know that $v = u$ on $\cI$, with $u$ as in \eqref{u-candidate}. Since $v\in C^2(\mathcal{I})$, then $v' \in C^1(\cI)$. Moreover, it is easy to check from \eqref{u-candidate} that $v'''$ is locally bounded on $\cI$; that is, $v' \in W^{2,\infty}_{loc}(\cI)$. Moreover, 
\begin{equation}
\label{1-OS}
\big(\cL_{\X} - (r-\mu'(x))\big)v'(x) + h'(x) =0 \quad \mbox{and} \quad -c_1(x) \leq v'(x) \leq c_2(x) \quad \mbox{on} \quad (a^*,b^*),
\end{equation}
\begin{equation}
\label{2-OS}
v'(x) = -c_1(x) < c_2(x)\quad \mbox{and} \quad \big(\cL_{\X} - (r-\mu'(x))\big)v'(x)+ h'(x)  \leq  0 \quad \mbox{on}\,\,(\underline{x}, a^*],
\end{equation}
\begin{equation}
\label{3-OS}
v'(x) = c_2(x) > -c_1(x) \quad \mbox{and} \quad \big(\cL_{\X} - (r-\mu'(x))\big)v'(x) + h'(x) \geq  0 \quad \mbox{on}\,\,[b^*,\overline{x}).
\end{equation}
The first equation in \eqref{1-OS} easily follows by noticing that $(\cL_{X} - r)v(x) + h(x) =0$ for any $x\in (a^*,b^*)$, and by differentiating such an equation with respect to $x$. On the other hand, the inequalities involving $(\cL_{\X} - (r-\mu'))v'+ h'$ in \eqref{2-OS} and in \eqref{3-OS} follow from Assumption \ref{A2}-(ii),  together with the fact that $(\underline{x}, a^*] \subset (\underline{x}, \widetilde{x}_1)$ and $[b^*,\overline{x}) \subset (\widetilde{x}_2,\overline{x})$.

Given an arbitrary $\widehat{\mathbb{F}}$-stopping time $\tau$, an application of a generalized version of It\^o's lemma (see, e.g., Theorem 10.4.1 in \cite{Oksendal}) to the process $(e^{-\int_0^{t}(r - \mu'(\X_u))du}v'(\widehat X_{t}))_{t\geq 0}$ on the interval $[0,\tau \land \sigma^*]$ yields $v'(x) \leq \widehat{\mathcal{J}}_x(\tau,\sigma^*)$, upon using \eqref{1-OS}-\eqref{3-OS}. On the other hand, applying again generalized It\^o's lemma to the process $(e^{-\int_0^{t}(r - \mu'(\X_u))du} v'(\widehat X_{t}))_{t\geq 0}$, but now on the interval $[0,\tau^* \land \sigma]$ for an arbitrary $\widehat{\mathbb{F}}$-stopping time $\sigma$, and employing \eqref{1-OS}-\eqref{3-OS} gives $v'(x) \geq \widehat{\mathcal{J}}_x(\tau^*,\sigma)$. Finally, by using generalized It\^o's lemma on the process $(e^{-\int_0^{t}(r - \mu'(\X_u))du} v'(\widehat X_{t}))_{t\geq 0}$ on the interval $[0,\tau^* \land \sigma^*]$ one finds that $v'(x) = \widehat{\mathcal{J}}_x(\tau^*,\sigma^*)$ by \eqref{1-OS}.

Hence \eqref{gameOS} holds true, and the $\widehat{\mathbb{F}}$-stopping times defined in \eqref{saddle} form a saddle-point.
\end{proof}

The previous proposition shows that $v'$ equals the value function of a zero-sum game of optimal stopping (a so-called Dynkin game, cf.\ \cite{Dynkin}). Furthermore, the boundaries that trigger the optimal control, also determine a saddle-point in the game of optimal stopping. Such finding is consistent to the known relation between bounded variation control problems and zero-sum games of optimal stopping (see, e.g., \cite{GT08}, \cite{KW}, and \cite{Taksar85}).


\section{A Case Study with a Mean-Reverting (Log-)Exchange Rate}
\label{sec:OU}

In this section we assume that in \eqref{stateX} one has $\mu(x) = \rho (m - x)$, for some $\rho>0$ and $m\in \R$, and $\sigma(x) = \sigma>0$; that is, for a given $\nu=\xi-\eta \in \mathcal{S}$, the (log-)exchange rate $X^{x;\xi,\eta}$ is a linearly controlled mean-reverting process with dynamics
\begin{equation}
\label{OU}
dX_t = \rho(m - X_t)\ dt + \sigma\ dB_t + d\xi_t - d\eta_t, \qquad X_0=x \in \R.
\end{equation}
In absence of interventions (i.e.\ $\nu\equiv 0$), this specification is the simplest dynamics which keeps $X$ in a given (suitable) region with a high probability, and empirical studies (see, e.g., \cite{Bo,Tvedt}) have concluded that it well describes several exchange rates among the main world countries.

In this section we also take the instantaneous costs $c_i$, $i=1,2$, such that $c_i(x) \equiv c_i$ for all $x \in \R$, and we specify a quadratic holding cost function of the form
$$h(x;\theta) = \frac12 (x - \theta)^2.$$
The parameter $\theta>0$ represents a so-called \emph{reference target}, and it can be also viewed as the logarithm of the \emph{central parity} (introduced in \cite{Krugman} for the first time). The function $h$ penalizes any displacement of the (log-)exchange rate from such a value.

We notice that with $X$ as in \eqref{OU}, we have $\sigma'(x) = 0$ for all $x \in \cI$. It thus follows that the process $\X$ of \eqref{statehatX} is the unique strong solution to 
\begin{equation}
\label{hatOU}
d\X_t = \rho(m - \X_t)\ dt + \sigma\ d\widehat{B}_t, \qquad X_0=x \in \R,
\end{equation}
where $\widehat{B}$ is a standard Brownian motion. Moreover, because $r-\mu'(x)=r + \rho$, the characteristic equation \eqref{ODE2} reads $\frac{1}{2}\sigma^2 f'' + \rho(m- x)f' - (r+\rho)u =0$, $r > 0$, and it is known that it admits the two linearly independent, positive solutions (cf.~\cite{JYC}, p.\ 280)
\begin{equation}
\label{phiOU}
\widehat{\phi}(x):=
e^{\frac{\rho(x-m)^2}{2\sigma^2}}D_{-\frac{r+\rho}{\rho}}\Big(\frac{(x-m)}{\sigma}\sqrt{2\rho}\Big)
\end{equation}
and
\begin{equation}
\label{psiOU}
\widehat{\psi}(x):=
e^{\frac{\rho(x-m)^2}{2\sigma^2}}D_{-\frac{r+\rho}{\rho}}\Big(-\frac{(x-m)}{\sigma}\sqrt{2\rho}\Big),
\end{equation}
which are strictly decreasing and strictly increasing, respectively. In both \eqref{phiOU} and \eqref{psiOU} $D_{\alpha}$ is the cylinder function of order $\alpha$ given by (see, e.g., \cite{Trascendental}, Chapter VIII, Section 8.3, eq.\ (3) at page 119)
\begin{align}
\label{cylinder}
D_{\alpha}(x):= \frac{e^{-\frac{x^2}{4}}}{\Gamma(-\alpha)}\int_0^{\infty}t^{-\alpha -1} e^{-\frac{t^2}{2} - x t} dt, \quad \text{Re}(\alpha)<0,
\end{align}
where $\Gamma(\,\cdot\,)$ is the Euler's Gamma function.

Within this setting, it is easy to see that the two equations for the free boundaries $a^*$ and $b^*$ (cf.\ \eqref{first} and \eqref{second}) read
\begin{align}
\int_a^b  \big(z-\theta + (r+\rho)c_1\big) \widehat m'(z) \widehat \phi(z)\ dz & =  -(r+\rho)(c_1+c_2)\int_b^{\infty} \widehat m'(z) \widehat \phi(z)\ dz, \label{first-OU} \\
\int_a^b  \big(z-\theta - (r+\rho)c_2 \big) \widehat m'(z) \widehat \psi(z)\ dz & =  (r+\rho)(c_1+c_2) \int_{-\infty}^a \widehat m'(z) \widehat \psi(z)\ dz, \label{second-OU}
\end{align}
where $\widehat{m}'$ is given by \eqref{hatm}. 

In the next section we study the dependency of the optimal boundaries $a^*$ and $b^*$ with respect to the model's parameters $m$, $\sigma$, $c_1$, $c_2$, and $\theta$. 

\subsection{Comparative Statics Results}
\label{sec:comparative}

In the following we will often use the notation $a^*(\,\cdot\,)$, $b^*(\,\cdot\,)$ and $v'(x;\cdot)$ to stress the dependence of $a^*$, $b^*$ and $v'$ with respect to a given parameter. For some of the next results (namely, Propositions \ref{prop:CSm}, \ref{prop:CSc}, and \ref{prop:CStheta}) an important role is played by the representation of $v'$ given in \eqref{gameOS}.

\begin{prop}
\label{prop:CSm}
The optimal intervention boundaries $a^*$ and $b^*$ are decreasing in the long-run equilibrium level $m$; that is, $m \mapsto a^*(m)$ and $m \mapsto b^*(m)$ are decreasing.
\end{prop}

\begin{proof}
Denote by $\X^{x;m}_t$ the unique strong solution to \eqref{hatOU} when the equilibrium value is $m\in \R$, and notice that $m \mapsto \X^{x;m}_t$ is a.s.\ increasing for all $t > 0$. Then, since $x \mapsto h'(x)$ is increasing, we have for all $m_1 \geq m_2$ that $v'(x;m_1) \geq v'(x;m_2)$ by \eqref{gameOS}.
Hence for all $m_1 \geq m_2$ we have
\begin{align*}
a^*(m_1) & = \sup\{ x \in \R\ |\ v'(x;m_1) \leq - c_1 \} \leq \sup\{ x \in \R\ |\ v'(x;m_2) \leq - c_1 \} = a^*(m_2),  \\
b^*(m_1) & =  \inf\{ x \in \R\ |\ v'(x;m_1) \geq c_2 \} \leq  \inf\{ x \in \R\ |\ v'(x;m_2) \geq c_2 \} = b^*(m_2);
\end{align*}
that is, $m \mapsto a^*(m)$ and $m \mapsto b^*(m)$ are decreasing.
\vspace{0.25cm}
\end{proof}

\begin{prop}
\label{prop:CSsigma}
The more the exchange market is volatile, the more the central bank is reluctant to intervene. That is, the optimal intervention boundaries $a^*$ and $b^*$ are such that
$\sigma \mapsto a^*(\sigma)$ is decreasing, and $\sigma \mapsto b^*(\sigma)$ is increasing.
\end{prop}

\begin{proof}
We borrow arguments from the proof of Theorem 6.1 in \cite{Matomaki} (see also Section 3 in \cite{Alvarez03}). Let $x\in \cI$, and denote by $X^{\sigma}$ the solution to \eqref{OU} for $\nu \equiv 0$ and for a volatility coefficient $\sigma>0$. Let $\sigma_1 \geq \sigma_2$, and for $i=1,2$ denote by $v^{(i)}$ the value function \eqref{value} when the underlying controlled process solves \eqref{OU} with volatility $\sigma_i$, by $\cL_{X^{\sigma_i}}$ the infinitesimal generator associated to the (uncontrolled) diffusion $X^{\sigma_i}$, and by $\cL_{\X^{\sigma_i}}$ the infinitesimal generator associated to the solution to \eqref{hatOU} with volatility $\sigma_i$. Also, $a^*_i$ and $b^*_i$, $i=1,2$, are the optimal control boundaries associated to the value function $v^{(i)}$, $i=1,2$.

Then, recall that the value function equals the function given in \eqref{u-candidate} and use that for any $i=1,2$ (cf.\ \eqref{arrange-final}),
\begin{equation}
\label{arrange-OU}
v^{(i)}(a^*_i)=\frac{1}{r} \left[ h(a^*_i) - \mu(a^*_i) c_1\right], \qquad
v^{(i)}(b^*_i)=\frac{1}{r} \left[ h(b^*_i) + \mu(b^*_i) c_2\right], 
\end{equation}
to find
\begin{equation*}
\big(\cL_{X^{\sigma_1}} - r\big)v^{(2)}(x) + h(x) := \left\{ \begin{array}{ll}
h(x) - h(a^*_2) + c_1(r+\rho)(x - a^*_2), & x \in (\underline x, a^*_2], \\
\\
\frac{1}{2}\big(\sigma_1^2 - \sigma_2^2)\frac{d^2}{dx^2}v^{(2)}(x), & x \in (a^*_2, b^*_2), \\
\\
h(x) - h(b^*_2) - c_2(r+\rho)(x - b^*_2), & x \in [b^*_2,\overline x). \\
\end{array} \right. 
\end{equation*}

We now prove that all the terms appearing on the right hand-side of the latter equation are nonnegative. On the one hand, 
$$h(x) - h(a^*_2) + c_1(r+\rho)(x - a^*_2) = \int_{x}^{a^*_2} \big[-h'(z) + \big(\cL_{\X^{\sigma_1}} - (r+\rho)\big)c_1\big] dz \geq 0, \quad x \in (\underline x, a^*_2],$$
and
$$h(x) - h(b^*_2) - c_2(r+\rho)(x - b^*_2)=\int_{b^*_2}^{x} \big[h'(z) + \big(\cL_{\X^{\sigma_1}} - (r+\rho)\big)c_2\big] dz \geq 0, \quad x \in [b^*_2,\overline x).$$
where the last inequalities are due to Assumption \ref{A2} and the fact that $a^*_2 < \widetilde{x}_1 < \widetilde{x}_2 < b^*_2$.

On the other hand, we notice that since $c_i(x)=c_i$ for all $x \in \cI$, the convexity of $x \mapsto h(x;\theta)$ and the linearity of the dynamics \eqref{OU} imply that functional $\mathcal{J}_x(\nu)$ is simultaneously convex in $(x,\nu)$, and the set of admissible controls is convex. Therefore, 
$$v^{(2)}(\lambda x + (1-\lambda)x') \leq \lambda \mathcal{J}_x(\nu) + (1-\lambda)\mathcal{J}_{x'}(\nu'),$$
for all $x,x' \in \cI$, $\nu\in \mathcal{A}(x)$, $\nu'\in \mathcal{A}(x')$, and $\lambda \in [0,1]$. Hence $v^{(2)}$ is convex on $\cI$, and this fact in turn yields
$$\frac{1}{2}\big(\sigma_1^2 - \sigma_2^2)\frac{d^2}{dx^2}v^{(2)}(x) \geq 0,\quad x \in (a^*_2, b^*_2),$$
since $\sigma_1^2 \geq \sigma_2^2$.

It thus follows from the previous considerations that $(\cL_{X^{\sigma_1}} - r)v^{(2)}(x) + h(x) \geq 0$ for all $x\in \cI$. Moreover, since $v^{(2)}$ is the value function when $\sigma=\sigma_2$, we also have $\frac{d}{dx}v^{(2)} \in [-c_1,c_2]$ and $|v^{(2)}(x)| \leq K(1 + |x|^{\gamma})$ on $\cI$, for some $K>0$, and for $\gamma \geq 1$ as in Assumption \ref{A2}. Therefore, arguing as in \emph{Step 1} of the proof of Theorem \ref{thm:verifico} we can show that $v^{(2)} \leq v^{(1)}$ on $\cI$.

Thanks to the last inequality we can now prove that $a^*_2\geq a^*_1$. We follow a contradiction scheme, and we suppose that $a^*_2 < a^*_1$. Then noticing that $a^*_2 < a^*_1 < \widetilde{x}_1$, and using \eqref{arrange-OU} and Assumption \ref{A2}-(ii) we have that (recall that here $\mu(x)=\rho(m-x)$, $x \in \mathbb{R}$)
\begin{align*}
v^{(2)}(a^*_2) & = \frac{1}{r}\big[h(a^*_2) - \mu(a^*_2)c_1\big] = \frac{1}{r}\big[h(a^*_2) + (ra^*_2 -\mu(a^*_2))c_1\big] -c_1a^*_2 \nonumber \\
& > \frac{1}{r}\big[h(a^*_1) + (ra^*_1 -\mu(a^*_1))c_1\big] -c_1a^*_2 = v^{(1)}(a^*_1) -c_1(a^*_2 -a^*_1) = v^{(1)}(a^*_2).
 \end{align*}
The latter inequality contradicts that $v^{(2)} \leq v^{(1)}$ on $\cI$, and therefore shows that $a^*_2\geq a^*_1$. Analogous arguments can be employed to obtain $b^*_1 \geq b^*_2$.
\end{proof}

\begin{prop}
\label{prop:CSc}
The optimal intervention boundaries $a^*$ and $b^*$ are such that
$c_1 \mapsto a^*(c_1)$ is decreasing, and $c_1 \mapsto b^*(c_1)$ is increasing. Also, $c_2 \mapsto a^*(c_2)$ is decreasing and $c_2 \mapsto b^*(c_2)$ is increasing.
\end{prop}

\begin{proof}
From \eqref{gameOS} and \eqref{functionalOS} it is easy to see that
\begin{align*}
c_1 & \mapsto  v'(x;c_1) + c_1 \mbox{ is increasing}, \qquad  & c_2 \mapsto  v'(x;c_2) - c_2 \mbox{ is decreasing} \\
c_1 & \mapsto  v'(x;c_1)  \mbox{ is decreasing}, \qquad & c_2 \mapsto  v'(x;c_2) \mbox{ is increasing}.
\end{align*}
Take now $\overline{c} > c_1$. Then
\begin{align*}
a^*(\overline{c}) & =  \sup\{ x \in \R\ |\ v'(x;\overline{c}) + \overline{c} \leq 0\} \leq \sup\{ x \in \R\ |\ v'(x;c_1) + c_1 \leq 0\} = a^*(c_1),  \\
b^*(\overline{c}) & =  \inf\{ x \in \R\ |\ v'(x;\overline{c}) \geq c_2 \} \geq  \inf\{ x \in \R\ |\ v'(x;c_1) \geq c_2 \} = b^*(c_1).
\end{align*}
Hence $c_1 \mapsto a^*(c_1)$ is decreasing and $c_1 \mapsto b^*(c_1)$ is increasing.

Analogously, taking now $\overline{c} > c_2$ we have
\begin{align*}
a^*(\overline{c}) & =  \sup\{ x \in \R\ |\ v'(x;\overline{c}) \leq - c_1 \} \leq \sup\{ x \in \R\ |\ v'(x;c_2) \leq - c_1 \} = a^*(c_2),  \\
b^*(\overline{c}) & =  \inf\{ x \in \R\ |\ v'(x;\overline{c}) - \overline{c} \geq 0\} \geq  \inf\{ x \in \R\ |\ v'(x;c_2) - c_2 \geq 0 \} = b^*(c_2);
\end{align*}
i.e., $c_2 \mapsto a^*(c_2)$ is decreasing and $c_2 \mapsto b^*(c_2)$ is increasing.
\end{proof}

\begin{prop}
\label{prop:CStheta}
The optimal intervention boundaries $a^*$ and $b^*$ are such that $\theta \mapsto a^*(\theta)$ and $\theta \mapsto b^*(\theta)$ are increasing.
\end{prop}

\begin{proof}
We notice that $\theta \mapsto h'(x;\theta)$ is decreasing. It follows from \eqref{gameOS} that $\theta \mapsto v'(x;\theta)$ is decreasing as well, and therefore for all $\theta_2 > \theta_1$ we have
\begin{align*}
a^*(\theta_1) & =  \sup\{ x \in \R\ |\ v'(x;\theta_1) \leq - c_1 \} \leq \sup\{ x \in \R\ |\ v'(x;\theta_2) \leq - c_1 \} = a^*(\theta_2),  \\
b^*(\theta_1) & =  \inf\{ x \in \R\ |\ v'(x;\theta_1) \geq c_2 \} \leq  \inf\{ x \in \R\ |\ v'(x;\theta_2) \geq c_2 \} = b^*(\theta_2).
\end{align*}
Thus, $\theta \mapsto a^*(\theta)$ and $\theta \mapsto b^*(\theta)$ are both increasing, so that when the target level $\theta$ increases, the no-intervention region $(a^*,b^*)$ is displaced towards higher values.
\end{proof}

\begin{remark}
Notice that the previous monotonicity results can be easily generalized to the case of a more general diffusion. For example, assuming that a comparison principle \`a la Yamada-Watanabe (see, e.g., Proposition 5.2.18 in \cite{KS}) holds true for the diffusion \eqref{statehatX}, that the killing rate $r-\mu'(\,\cdot\,)$ is decreasing (i.e.\ $x \mapsto \mu(x)$ is convex), and that the holding cost function is convex, one can show that the boundaries are monotonically decreasing with respect to the drift coefficient. Also, arguing as in the proof of Theorem 6.1 in \cite{Matomaki}, one can prove the monotonicity of the boundaries with respect to a general state-dependent volatility coefficient. However, we decided to formulate the study of this section for an Ornstein-Uhlenbeck process because it is perhaps the simplest diffusion that captures the mean-reverting behavior of exchange rates empirically observed in some economies (see \cite{Sweeney, Tvedt, Yang}, and references therein). 
\end{remark}


\subsection{Expected Exit Time from the Target Zone}
\label{sec:exittime}

One of the great advantage of the Ornstein-Uhlenbeck model above is that many quantities about exit times and probabilities are known in closed form. We base our analysis on the results contained in \cite[Appendix B]{CadSarZap}. Recalling the optimally controlled process $X^{\star}$, define the exit time from $(a^*,b^*)$ as
$$ \tau_{(a^*,b^*)} := \inf\{ t > 0\ :\ X^{\star}_{t} \notin (a^*,b^*) \}, $$
and notice that $\P_x\{\tau_{(a^*,b^*)} < \infty\} = 1$ for all $x \in \R$. Indeed, if $x \notin (a^*,b^*)$ then clearly $\tau_{(a^*,b^*)} = 0$ $\P_x$-a.s. On the other hand, if $x \in (a^*,b^*)$ then the optimal control $\nu^{\star}$ is such that $\nu^{\star}_t \equiv 0$ for any $t \leq \tau_{(a^*,b^*)}$, and the (uncontrolled) Ornstein-Uhlenbeck process is positively recurrent. Also, we have that
\begin{eqnarray}
\P_x\{ X_{\tau_{(a^*,b^*)}} = a^* \} & = & \frac{ \displaystyle \int_x^{b^*} \exp\left( \rho \frac{(y - m)^2}{\sigma^2} \right)\ dy}{\displaystyle \int_{a^*}^{b^*} \exp\left( \rho \frac{(y - m)^2}{\sigma^2} \right)\ dy} = \frac{\displaystyle \int_{\frac{\sqrt{2\rho}}{\sigma}(x-m)}^{\frac{\sqrt{2\rho}}{\sigma}(b^*-m)} e^{\frac12 y^2} \ dy}{\displaystyle \int_{\frac{\sqrt{2\rho}}{\sigma}(a^*-m)}^{\frac{\sqrt{2\rho}}{\sigma}(b^*-m)} e^{\frac12 y^2} \ dy}, \label{Pa} \\ \nonumber \\
\P_x\{ X_{\tau_{(a^*,b^*)}} = b^* \} & = & \frac{\displaystyle \int_{a^*}^x \exp\left( \rho \frac{(y - m)^2}{\sigma^2} \right)\ dy}{\displaystyle \int_{a^*}^{b^*} \exp\left( \rho \frac{(y - m)^2}{\sigma^2} \right)\ dy} = \frac{\displaystyle \int_{\frac{\sqrt{2\rho}}{\sigma}(a^*-m)}^{\frac{\sqrt{2\rho}}{\sigma}(x-m)} e^{\frac12 y^2} \ dy}{\displaystyle \int_{\frac{\sqrt{2\rho}}{\sigma}(a^*-m)}^{\frac{\sqrt{2\rho}}{\sigma}(b^*-m)} e^{\frac12 y^2} \ dy}.\label{Pb}
\end{eqnarray}
Furthermore, we know that the function $q(x) := \E_x[\tau_{(a^*,b^*)}]$, $x \in (a^*,b^*)$, satisfies the boundary value differential problem 
$$ {\mathcal L}_X q + 1 = 0, \qquad q(a) = q(b) = 0, $$
whose solution is
\begin{equation} \label{exittime2}
q(x) = A_1 + B_1 \int_{\frac{\sqrt{2\rho}}{\sigma}(a^*-m)}^{\frac{\sqrt{2\rho}}{\sigma}(x-m)} e^{\frac12 w^2}\ dw - \frac{1}{\rho} \int_{\frac{\sqrt{2\rho}}{\sigma}(y-m)}^{\frac{\sqrt{2\rho}}{\sigma}(b^*-m)} e^{\frac12 w^2} \int_w^{\frac{\sqrt{2\rho}}{\sigma}(b^*-m)} e^{-\frac12 u^2}\ du \ dw, 
\end{equation}
with the constants $A_1$ and $B_1$ given by
$$ A_1 = \frac{1}{\rho} \int_{\frac{\sqrt{2\rho}}{\sigma}(a^*-m)}^{\frac{\sqrt{2\rho}}{\sigma}(b^*-m)} e^{\frac12 w^2} \int_w^{\frac{\sqrt{2\rho}}{\sigma}(b^*-m)} e^{-\frac12 u^2}\ du \ dw, \qquad B_1 = \frac{- A_1}{\displaystyle \int_{\frac{\sqrt{2\rho}}{\sigma}(a^*-m)}^{\frac{\sqrt{2\rho}}{\sigma}(b^*-m)} e^{\frac12 w^2}\ dw}. $$
Thanks to the previous results we can numerically compute the mean time until the exchange rate leaves the target zone, i.e.\ the mean time until the next central bank's intervention. This will be done in the next section. 


\subsection{Numerical results}

We now present a possible implementation of the previous model, tailored to mimick the DKK/EUR exchange rate. Since it seems that in 30 years there was no need to intervene from the Danish Central Bank, we can safely assume that the long-run mean corresponds to the logarithm of the central parity fixed to 7.46038 DKK/EUR. Remembering that the Ornstein-Uhlenbeck process in equation \eqref{OU} represents the logarithm of the exchange rate, we thus let $m = \theta = \log 7.46038 = 2.00961 = \simeq 2.01$; other plausible parameters for the Ornstein-Uhlenbeck dynamics could be $\rho = 0.001$ and $\sigma = 0.015$. Given the interest rates in the current economy, a plausible value for $r$ could be $r = 0.5\% = 0.005$. 
The values above are characteristic of the Danish and European economies, and still do not reflect the Danish Central Bank's policy, which is instead implemented in the three parameters $\theta$, $c_1$ and $c_2$. We collect the parameters up to now in Table \ref{table-OU}.

\begin{table}[h]
	\centering
	\begin{tabular}   {| l | l | l | l | l | 
	}
		\hline
		$r$ & $\sigma$ & $\rho$ & $\theta$ & $m$ 
		\\ \hline
		0.005 & 0.015 & 0.001 & 2.01 & 2.01 
		\\
		\hline
	\end{tabular}
	\vspace{0.25cm}
	\caption{Parameters' values for the numerical example.}
	\label{table-OU}
\end{table} 

In order to find the known intervention thresholds of $\pm 2.25\%$ from the central parity, we must implement the following 
inverse problem: find $c_{1}, c_2$ such that, with the parameters above, the optimal $a^*$ and $b^*$ are 
$$ a^* = \log  7.46038(1 - 0.0225) = 1.98685
, \qquad b^* = \log 7.46038 (1 + 0.0225) =  2.03186
$$

Given the (approximate) symmetry of our problem\footnote{since $\log (1 + 0.0225) = 0.02225 \simeq 0.0225 \simeq - 0.02276 = - \log (1 - 0.0225)$}, we search for $c_{1}$ and $c_2$ such that $c_1 = c_2 =: c$. From the monotonicity result of Proposition \ref{prop:CSc} we know that, by increasing (decreasing) the common proportional cost $c$, the continuation region $(a^*,b^*)$ will enlarge (shrink): this is a positive sign that our inverse problem can have a unique solution. 

With this in mind, we search for $c = c_1 = c_2$ such that $[a^*,b^*] \simeq [m - 0.0225, m + 0.0225]$. We start checking for $c = 1$, and we continue by decreasing the value of $c$ until we find our zone: the results are reported in Table \ref{table-c}.

\begin{table}[h]
\centering
\begin{tabular}{l|llll}
$c$			& $a^*$		& $b^*$		& $a^* - m$	& $b^* - m$\\
\hline
1			& 1.93729	& 2.08193	& $-$0.07232	& 0.07232\\
0.5			& 1.95302	& 2.0662	& $-$0.0565905 	& 0.0565905\\
0.1			& 1.97703	& 2.04218	& $-$0.0325786	& 0.0325786\\
0.05			& 1.98383	& 2.03539	& $-$0.0257803	& 0.0257803\\
0.04			& 1.98569	& 2.03352	& $-$0.0239155	& 0.0239155\\
0.035		& 1.98674	& 2.03247	& $-$0.0228658	& 0.0228658\\
0.034		& 1.98696	& 2.03225	& $-$0.0226442	& 0.0226442\\
{\bf 0.0335}		& {\bf 1.98707}	& {\bf 2.03214}	& {\bf $-$0.0225317}	& {\bf 0.0225317}\\
0.033		& 1.98719	& 2.03202	& $-$0.0224182	& 0.0224182\\
0.03			& 1.98789	& 2.03132	& $-$0.021712	& 0.021712\\
\end{tabular}
	\vspace{0.25cm}
	\caption{Optimal continuation regions (i.e., target zones) for various fixed costs $c = c_1 = c_2$.}
	\label{table-c}
\end{table}

Hence, given the parameters' values of Table \ref{table-OU}, if we let $c_1 = c_2 = 0.0335$, we find that the optimal $a^*$ and $b^*$ well approximate the boundaries of the target zone that the Central Bank of Denmark is adopting since January 12, 1987 \cite{Danish}. 

By using the results in Section \ref{sec:exittime}, we can also compute the expected exit time of the exchange rate from the target zone. In fact, by taking $(a^*,b^*) = (1.98707,2.03214)$, equation \eqref{exittime2} can be plotted as a function of initial (log-)exchange rate $x$. The plot is in Figure \ref{fig:exittime} below.

\begin{figure}[h!]
\centering
\includegraphics[scale=0.8]{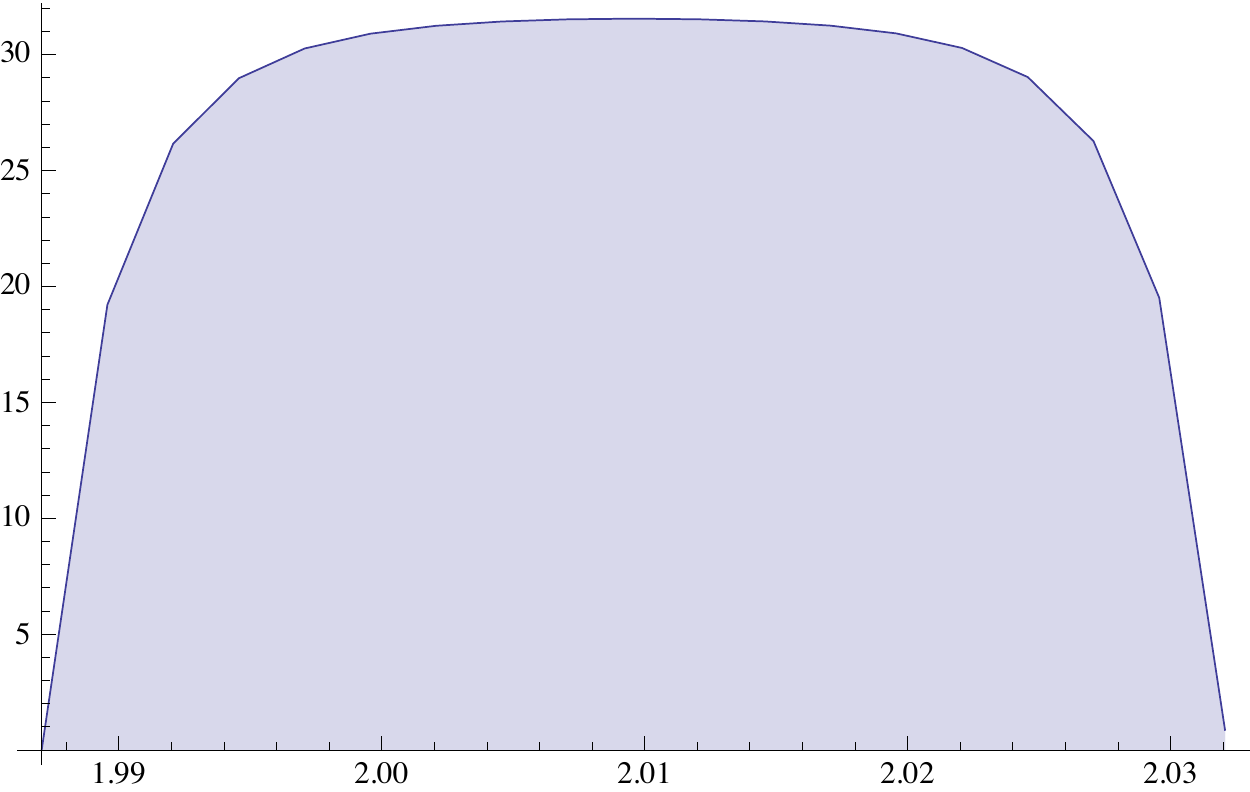}
\caption{Average exit time from the target zone (in years), as a function of the initial value $x$ of (the logarithm of) the exchange rate.}
\label{fig:exittime}
\end{figure}

We can see that the maximal expected time is obtained (as expected) when the deviation from central parity is null, i.e., for $x = \log  7.46038 \simeq 2.01$, and decreases as the exchange rate nears the target zone's boundaries. This maximum expected time is around 31.11 years, which is also the expected time before an intervention by the central bank is triggered. This finding is perfectly in line with the observed phenomenon that the Danish Central Bank did not need to intervene to keep the DKK/EUR exchange rate within the target zone since the last 30 years \cite{Danish}.

\medskip

We now try to reproduce, on this set of parameters, the ``pegging" phenomenon which was observed in the CHF/EUR exchange rate in the period 2011--2015 (see Figure \ref{EURCHF} and \cite{NYT,Economist}). The economic intuition behind it is that the central bank would intervene in pegging the rate above (or below) a certain threshold, even if the rates' uncontrolled dynamics would push it beyond that threshold. This can be easily implemented in the present framework, by simply changing $\theta$ to be a value different from $m$. Due to the monotonicity results of Proposition \ref{prop:CStheta}, this would modify the width of the continuation region $(a^*,b^*)$.  If for example, in the previous framework, we let $\theta = m + \delta$ for $\delta\geq 0$, then we expect that $a^*$ and $b^*$ are increased. In fact, by keeping all the other parameters fixed, we find the results in Table \ref{table-delta}.

\begin{table}[h]
\centering
\begin{tabular}{l|llll}
$\delta = \theta - m$	& $a^*$		& $b^*$		& $a^* - m$		& $b^* - m$\\
\hline
0						& 1.98707	& 2.03214	& $-$0.02253		& 0.0225317\\
0.01					& 1.99709	& 2.04215	& $-$0.01251	& 0.0325466\\
0.02						& 2.0071		& 2.05217	& $-$0.00250 	& 0.0425615\\

0.03						& 2.01712	& 2.06218	& \ \ 0.00751		& 0.0525764\\
\end{tabular}
	\vspace{0.25cm}
	\caption{Optimal continuation regions (i.e., target zones) for various deviations of the (logarithm of the) central parity $\theta$ from the long-run mean $m$ of the uncontrolled exchange rate.}
	\label{table-delta}
\end{table}

We can observe that, by increasing $\theta$, also the target zone $(a^*,b^*)$ increases. In particular, it suffices to increase $\theta$ by $\delta = 0.03$ (remembering that $m \simeq 2.01$, this would amount to less than 2\% of relative increase of the exchange rate) to make that the long-term mean of the exchange rate is {\em outside} of the target zone (in fact, in this case we have $m = 2.01 < a^* = 2.01712	< b^* = 2.06218$). Even in the less severe scenario of $\delta = 0.02$, we would have that $m$ is still inside the target zone, but very near to the lower threshold. By using again the results in Section \ref{sec:exittime}, we can estimate the expected exit time of the exchange rate from the target zone also in this case, and plot the results in Figure \ref{fig:exitnear}.

\begin{figure}[h!]
\centering
\includegraphics[scale=0.8]{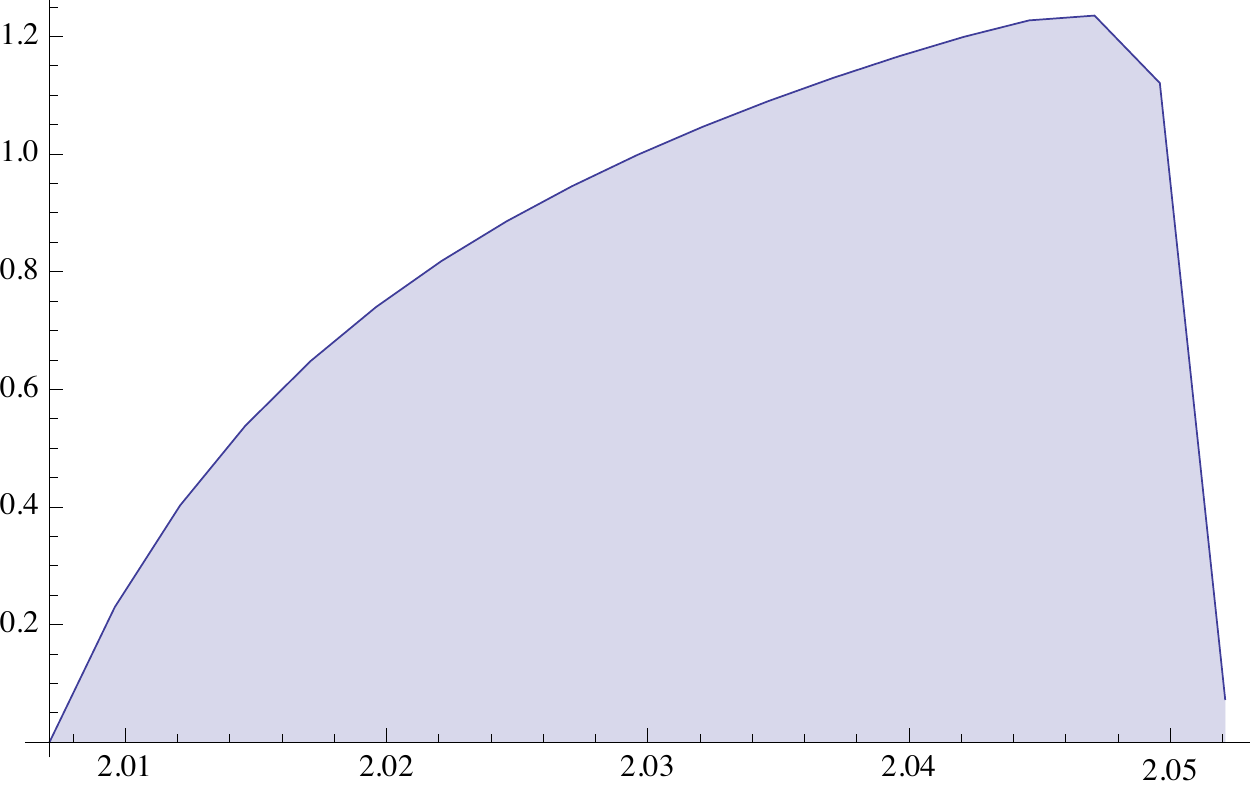}
\caption{Average exit time from the target zone (in years), as a function of the initial value $x$ of (the logarithm of) the exchange rate when $\theta - m = 0.02$.}
\label{fig:exitnear}
\end{figure}

As expected, here we notice a sharp asymmetry: the long-term mean is still $m \simeq 2.01$, but the target zone is now not symmetric around it. As a consequence, if the (log-)exchange rate starts from $x = m$, the expected exit time from the target zone (i.e.\ before the central bank is forced to intervene) is $q(2.01) = 0.23$, i.e.\ about 3 months (instead of the 30+ years of the previous case). However, if we start from $x > 2.01$, then the process will revert with high probability towards its mean $m$, taking some time in doing that, and then it will spend about other 3 months before hitting one of the boundaries of the target zone. Such an expected time to return to the long-term mean (and thus the expected exit time from the target zone) is an increasing function of the initial level $x$ of the (log-)exchange rate for any $x \leq x_o \simeq 2.048$. Letting the exchange rate process start from a value $x \geq x_o$, it then becomes more probable to exit the target zone from $b^*$, and the expected exit time starts to decrease. We can also see that at such a critical level we have $q(2.048) \simeq 1.26$; i.e.\ starting from $x_o$ the expected exit time is maximal, and it is about 1 year and 3 months. 

In Figure \ref{exitprob}, we draw the exit probability from $a^*$ as a function of initial state $x$. We can see that the probability of hitting the ``peg" $a^*$ is essentially equal to 1 for any initial value of the (log-)exchange rate $x < 2.045$. For higher values such a probability then starts to decrease, up to a critical point near to $x_o$ (actually, slightly above 2.05), where it becomes more probable to leave the target zone from $b^*$ than from $a^*$.

\begin{figure}[h!]
\centering
\includegraphics[scale=0.8]{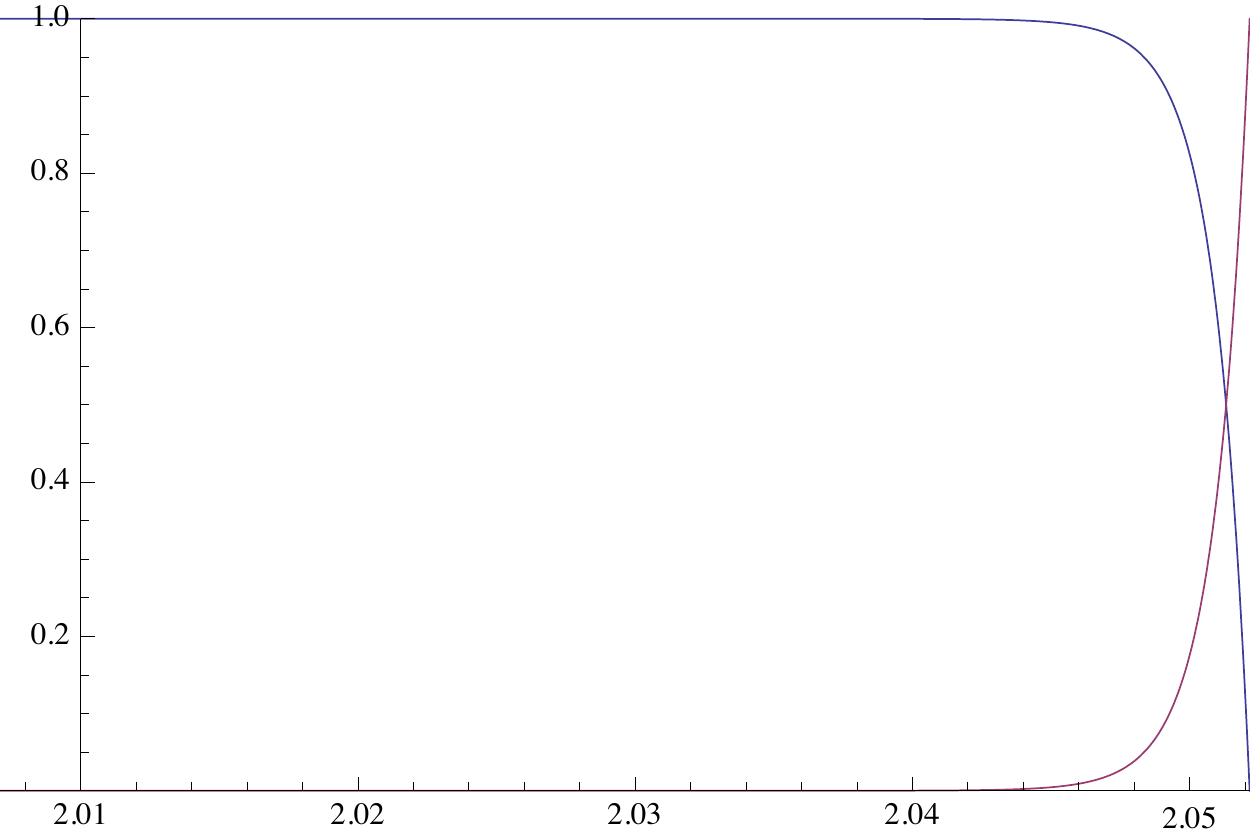}
\caption{Probabilities of exiting the target zone from $a^*$ (in blue) and from $b^*$ (in red), as functions of the initial value $x$ of (the logarithm of) the exchange rate when $\theta - m = 0.02$.}
\label{exitprob}
\end{figure}


\section{Concluding Remarks}

In this paper we have studied the optimal management problem of exchange rates faced by a central bank. We have formulated it as an infinite time-horizon singular stochastic control problem for a one-dimensional diffusion that is linearly controlled through a process of bounded variation. We have provided the explicit expression of the value function, as well as the complete characterization of the optimal control. At each instant of time, the optimally controlled exchange rate is kept within an optimal band (continuation region), whose boundaries (the so-called free boundaries) are endogenously determined as part of the solution to the problem. 

A detailed comparative statics analyisis of the free boundaries is provided when the (log-)exchange rate (in absence of any intervention) evolves as an Ornstein-Uhlenbeck process. This dynamics captures the mean-reverting behavior of exchange rates that has been observed in several empirical studies (see \cite{Sweeney, Tvedt} and references therein). Moreover, it allows the central bank to have aims, both in its cost function $h$ as well as in its intervention costs $c_i$, which possibly contrast with this foreign exchange dynamics. This does not happen if, for example, the minimum of $h$ is very near to the long-term mean of the exchange dynamics: in this case, the exchange rate stays naturally with a high probability in the continuation region. This fact can be interpreted as the ``target zone" introduced in \cite{Krugman}, and it applies, for example, to the Danish and Hong Kong currencies \cite{Danish,Wiki}. Instead, if the rate's long-term mean is far from the minimum of $h$, or worse even outside the continuation region, then it is very probable that the exchange rate hits one boundary of the continuation region much more often than the other one. This phenomenon is usually referred to as ``pegging" the exchange rate above or below a given threshold, and it has been observed in the period 2011--2015 in the dynamics of the Swiss Franc versus the Euro \cite{NYT,Economist}.

Several comments deserve to be made on our model, and on its possible extensions. First of all, it is worth noting that, given its generality, the control problem studied in this paper might be a reasonable model also in other context, as, e.g., for problems of partially reversible capacity expansion (see \cite{DeAF2013}, \cite{GP}, among others), for the optimal management of an inventory (see \cite{HT} for an early work), for the automotive cruise control of an aircraft under an uncertain wind condition \cite{CMR}, or for the optimal management of stabilization funds \cite{Aguilar}. Second of all, there are several possible directions towards our study on exchange rates' control can be extended. In particular, it would be extremely interesting to develop a mathematical model taking into account the strategic interaction between two (or more) central banks for the management of the exchange rates (see also \cite{aidetal} for a recent contribution in this direction). This would lead to a challenging nonzero-sum stochastic game with singular controls that we leave for future research.


\section*{Acknowledgments}
\noindent Financial support by the German Research Foundation (DFG) through the Collaborative Research Centre 1283 ``Taming uncertainty and profiting from randomness and low regularity in analysis, stochastics and their applications'' is gratefully acknowledged by the first author. Part of this work has been done while the first author was visiting the Department of Mathematics of the University of Padova thanks to the funding provided by the ``ACRI Young Investigator Training Program'' (YITP-QFW2017). 


\appendix
\label{sec:app}

\section{}
\renewcommand{\theequation}{A-\arabic{equation}}

\begin{lemma}
\label{lem:AM}
Under Assumption \ref{ass:psiprimephiprime} one has that $\psi'=\widehat{\psi}$ and $-\phi'=\widehat{\phi}$, where $\widehat{\psi}$ and $\widehat{\phi}$ are the strictly increasing and strictly decreasing fundamental solutions of the ODE $(\cL_{\X} - (r-\mu'))f =0$ for $\X$ killed at rate $r-\mu'(\,\cdot\,)$.
\end{lemma}

\begin{proof}
We simply repeat the arguments in the second part of the proof of Lemma 4.3 in \cite{AlvarezMatomaki} (see also Theorem 9 in \cite{Alvarez01}). Under Assumption \ref{A1} standard differentiation reveals that $\psi'$ and $\phi'$ solve the ODE 
\begin{equation}
\label{ODE-AM}
(\cL_{\X} - (r-\mu'))f =0.
\end{equation}
Also, for any $x\in \cI$ one has $\phi''(x)\psi'(x) - \phi'(x)\psi''(x)=2rW\S'(x)\neq 0$, and so any solution $f$ to the previous ODE has to be of the form $f(x)= c_1 \psi'(x) + c_2\phi'(x)$. Furthermore, note that under Assumption \ref{A1} and \ref{ass:rate}, Corollary 1 of \cite{Alvarez03} can be applied yielding that $\phi$ and $\psi$ are strictly convex.

We thus find that for all $\ell_1 < \ell_2$ and for all $x \in (\ell_1,\ell_2) \subset \mathcal{I}$ we can write
$$\widehat{\E}_x\Big[e^{-\int_0^{\widehat{\tau}}(r-\mu'(\X_s))ds}\Big] = \frac{f_1(x)}{f_1(\ell_1)} + \frac{f_2(x)}{f_2(\ell_2)},$$
where $\widehat{\tau}:= \inf\{t\geq 0: \X_t \notin (\ell_1,\ell_2)\}$, $\widehat{\P}_x$-a.s., and $f_1(x):=\frac{\phi'(\ell_2)}{\psi'(\ell_2)}\psi'(x) - \phi'(x)$ and $f_2(x):=\psi'(x) - \frac{\psi'(\ell_1)}{\phi'(\ell_1)}\phi'(x)$ are the fundamental solutions of \eqref{ODE-AM} when $\X$ is killed at $\ell_1$ and $\ell_2$.

Noticing that $\lim_{\ell_1 \downarrow \underline{x}}\psi'(\ell_1)/\phi'(\ell_1) = 0$ and $\lim_{\ell_2 \uparrow \overline{x}}\phi'(\ell_2)/\psi'(\ell_2) = 0$ by the required boundary behavior of $X$, Assumption \ref{ass:psiprimephiprime} implies that
$$\lim_{\ell_1 \downarrow \underline{x}}\widehat{\E}_x\Big[e^{-\int_0^{\widehat{\tau}}(r-\mu'(\X_s))ds}\Big] = \frac{\psi'(x)}{\psi'(\ell_2)},$$
and
$$\lim_{\ell_2 \uparrow \overline{x}}\widehat{\E}_x\Big[e^{-\int_0^{\widehat{\tau}}(r-\mu'(\X_s))ds}\Big] = \frac{\phi'(x)}{\phi'(\ell_1)}.$$
Hence, $\psi'$ and $-\phi'$ are the fundamental solutions of \eqref{ODE-AM} for $\X$ killed at rate $r-\mu'(\,\cdot\,)$. That is, $\psi'=\widehat{\psi}$ and $-\phi'=\widehat{\phi}$.
\end{proof}
\vspace{0.5cm}

\emph{Proof of Equation \eqref{relation-resolvents}}
\vspace{0.3cm}

Assumption \ref{A1} guarantees that the flow $x \mapsto X^{x;0,0}_t$ is a.s.\ continuous, increasing and differentiable for any $t\geq0$ (see, e.g., \cite{Protter}, Ch.\ V.7). Defining the process $Y$ such that $Y_t=\partial X^{x;0,0}_t/\partial x$, $t\geq 0$, by ordinary differentiation we find that $Y$ satisfies
$$dY_t = \mu'(X^{x;0,0}_t)Y_t dt + \sigma'(X^{x;0,0}_t)Y_t dB_t, \qquad Y_0=1,$$
and therefore that
$$ \displaystyle Y_t = e^{\int_0^t \mu'(X^{x;0,0}_s)ds}Z_t, $$
where the exponential martingale $(Z_t)_{t\geq 0}$ has been defined in equation \eqref{zeta}.

As in Remark \ref{rem:Girsa}, consider the dynamics of $X^{0,0}$ under the measure $\mathbb{P}_x$, and the dynamics of $\X$ under the measure $\widehat{\ppp}_x$. Define also a new measure $\mathbb{Q}_x$ through the Radon-Nikodym derivative $Z_t:=\frac{d\mathbb{Q}_x}{d\ppp_x}|_{\mathcal{F}_t}$\, and notice that the Girsanov theorem implies that the process 
$$\widehat{B}_t:=B_t-\int_0^t\sigma'(X^{0,0}_s)ds $$
is a standard Brownian motion under $\mathbb{Q}_x$.

Take now $f \in C^1(\cI)$, and such that $Rf$ and $\widehat{R}f'$ are well defined. Then, by differentiating \eqref{resolvent} we obtain
\begin{align*}
(Rf)'(x) & = \E_x\bigg[\int_0^{\infty} e^{- r t}\, Y_t\,f'(X^{0,0}_t) dt \bigg] \nonumber \\
& = \E^{\mathbb{Q}_x}\bigg[\int_0^{\infty} e^{-\int_0^t (r-\mu'(X^{0,0}_s)) ds} \,f'(X^{0,0}_t) dt \bigg].
\end{align*}
\noindent We therefore conclude by observing that $\text{Law}\,(X^{0,0}\big|\mathbb{Q}_x)=\text{Law}\,(\X\big|\widehat{\ppp}_x)$ and recalling \eqref{resolventhat}.


\end{document}